\documentclass[11pt, twoside]{article}
\usepackage{amsfonts}
\usepackage{mathrsfs}
\usepackage{amssymb}
\usepackage{amsmath,amsthm}
\usepackage{color}

\allowdisplaybreaks

\def\rr{{\mathbb R}}
\def\rn{{{\rr}^n}}

\def\cc{{\mathbb C}}
\def\nn{{\mathbb N}}
\def\zz{{\mathbb Z}}

\def\cl{{\mathcal L}}
\def\cm{{\mathcal M}}

\def\cp{\mathcal D}
\def\cq{{\mathcal Q}}
\def\crz{{\mathcal R}}

\def\noz{\nonumber}
\def\fz{\infty}
\def\az{\alpha}

\def\supp{{\mathop\mathrm{\,supp\,}}}

\def\loc{{\mathop\mathrm{\,loc\,}}}

\def\BMO{{\mathop\mathrm{\,BMO\,}}}

\def\lz{\lambda}

\def\dz{\delta}

\def\ez{\epsilon}

\def\kz{{\kappa}}

\def\sz{\sigma}
\def\wz{\widetilde}

\def\ls{\lesssim}

\def\boz{\Omega}

\def\pz{{\prime}}

\def\gfz{\genfrac{}{}{0pt}{}}
\def\hs{\hspace{0.3cm}}

\def\dint{\displaystyle\int}
\def\dfrac{\displaystyle\frac}

\def\r{\right}
\def\lf{\left}

\newtheorem{thm}{Theorem}[section]
\newtheorem{lem}{Lemma}[section]
\newtheorem{prop}{Proposition}[section]
\newtheorem{rem}{Remark}[section]

\newtheorem{defn}{Definition}[section]

\numberwithin{equation}{section}

\begin{document}

\arraycolsep=1pt

\title{{\vspace{-5cm}\small\hfill\bf J. Fourier Anal. Appl. (to appear)}\\
\vspace{4.5cm}\bf\Large Equivalent Characterizations for Boundedness
of Maximal Singular Integrals on $ax+b$\,--Groups
\footnotetext{\hspace{-0.22cm}
The first author was supported by the
Fundamental Research Funds for the Central Universities, and the
Research Funds of Renmin University of China (Grant No. 10XNF090).
The second author is partially supported by PRIN 2007 ``Analisi Armonica".
The third (corresponding) author is supported by the National
Natural Science Foundation (Grant No. 10871025) of China
and Program for Changjiang Scholars and Innovative
Research Team in University of China.
\endgraf $^\ast$Corresponding author.}}
\author{Liguang Liu, Maria Vallarino, and Dachun Yang $^\ast$}
\date{ }
\maketitle

\noindent{\bf Abstract}\quad Let $(S, d, \rho)$ be the affine
group $\mathbb{R}^n \ltimes \mathbb{R}^+$ endowed with the
left-invariant Riemannian metric $d$  and the right Haar measure
$\rho$, which is of exponential growth at infinity. In this paper,
for any linear operator $T$ on $(S, d, \rho)$ associated with a
kernel $K$ satisfying certain integral size condition and
H\"ormander's condition, the authors prove that the following four
statements regarding the corresponding maximal singular integral
$T^\ast$ are equivalent: $T^\ast$ is bounded from  $L_c^\infty$ to
$\mathrm{BMO}$, $T^\ast$ is bounded on $L^p$ for all $p\in(1,
\infty)$, $T^\ast$ is bounded on $L^p$ for some $p\in(1, \infty)$
and $T^\ast$ is bounded from $L^1$ to $L^{1,\,\infty}$. As
applications of these results, for spectral multipliers of a
distinguished Laplacian on $(S, d, \rho)$ satisfying certain
Mihlin-H\"ormander type condition, the authors obtain that their
maximal singular integrals are bounded from $L_c^\infty$ to
$\mathrm{BMO}$, from $L^1$ to $L^{1,\,\infty}$, and on $L^p$ for all
$p\in(1, \infty)$.

\bigskip


\bigskip


\vspace{0.5cm}

\section{Introduction}\label{s1}

Let $S$ be the {\em affine group} $\rn\ltimes \rr^+$
endowed with the {\em product}
$$(x, a)\cdot(x', a')=(x+ax', aa') \qquad\forall\ (x, a), (x', a')\in
S.$$
We call $S$ an $ax+b$\,--{\em group}.
Clearly, $e \equiv  (0, 1)$ is the {\em unit} of $S$. The {\em inverse} of any
$(x, a)\in S$, denoted by $(x, a)^{-1}$, is equal to $(-x/a, 1/a)$.
We endow $S$ with the {\em left-invariant Riemannian metric}
$ds^2 \equiv  a^{-2}(dx_1^2+\cdots+dx_n^2+da^2),$
which induces the following {\em distance function} $d$ on $S\times S$:
\begin{equation}\label{1.1}
d\big((x, a), (x', a')\big)= \cosh^{-1}
\lf(\dfrac{a^2+{a'}^2+|x-x'|^2}{2aa'}\r) \quad\forall\  (x, a), (x',
a')\in S;
\end{equation}
see, for example, \cite[formula (2.18)]{ADY} {or} \cite[pp.\,119--120]{mt}.
The {\em left Haar measure} induced by the
above left-invariant Riemannian metric is $d\lz(x, a)=a^{-n-1}dxda$.
A standard calculation yields that the {\em group modular
function} $\dz(x, a)$ is equal to $a^{-n}$, and hence the right Haar
measure is
\begin{equation}\label{1.2}
d\rho(x, a)=\dz(x, a)^{-1}d\lz(x, a)=a^{-1}dxda.
\end{equation}
Throughout the whole paper, we work on the {\em triple $(S, d, \rho)$},
namely, the group $S$ endowed with the left-invariant Riemannian
metric $d$  and the right Haar measure $\rho$. For all $(x, a)\in S$
and $r>0$, we define the {\em ball} $B\big((x,a), r\big)$ on $S$ by
$$B\big((x,a), r\big) \equiv \lf\{(x', a')\in S:\,d\big((x, a), (x',a') \big)<r\r\}.$$
For any $p\in(0, \fz]$, let $L^p$ be the {\em space} of all
measurable functions $f$ on $(S, d, \rho)$ satisfying
\begin{equation*}
\|f\|_{L^p} \equiv \lf\{\dint_S |f(x)|^p\,d\rho(x)\r\}^{1/p}<\fz\,,
\end{equation*}
with a usual modification made when $p=\fz$, and let $L^{p,\,\fz}$
be the {\em space} of all measurable functions $f$ on $(S, d,
\rho)$ satisfying
\begin{equation*}
\|f\|_{L^{p,\, \fz}}  \equiv \sup_{\az>0}\lf\{\az\lf[\rho\lf(\{x\in
S:\, |f(x)|>\az\}\r)\r]^{1/p}\r\}<\fz.
\end{equation*}
Denote by $L_c^\fz$ the {\em set} of  functions in $L^\fz$
with compact support; and  by $L_{c,0}^\fz$ the {\em set} of
functions $f$ in $L_c^\fz$ such that $\int_S f\,d\rho=0$.
Let $L_\loc^1$ be the {\em set} of locally integrable functions on $S$
(with respect to the measure $\rho$).

It is well known that the space $(S, d, \rho)$ is of {\em exponential
growth at infinity}, namely,
\begin{eqnarray*}
\rho\Big(B\big((0, 1), r\big)\Big)\approx\left\{\begin{array}{ll}
r^{n+1}\quad\quad&\text{if\quad $r<1$},\\
e^{nr}&\text{if\quad $r\ge 1$};
\end{array}\right.
\end{eqnarray*}
see, for example, \cite{mt, va}.
Harmonic analysis on exponential growth
groups including the space $(S, d, \rho)$ currently attracts a lot of
attention. In particular, efforts have been made to study the
Hardy--Littlewood maximal function  (see \cite{gs, gghm}), Riesz transforms (see
\cite{sv, gas, gqs, S99}), spectral multipliers related to a
distinguished Laplacian (see \cite{Astengo, cghm, hs, H98, H97,
mt, va, va2}).

In 2003, Hebisch and Steger \cite{hs} introduced
the notion of {\it Calder\'on--Zygmund spaces},
and developed a variant of the Calder\'on--Zygmund
theory of singular integrals on
Calder\'on--Zygmund spaces, which can be applied to the $ax+b$\,--groups.
Precisely, a {\em space} $M$ endowed with a {\em metric} $d$ and a {\em Borel measure} $\mu$ is said to have the
{\em Calder\'on--Zygmund property} if there exists a positive constant $C$
such that for every $f\in L^1(\mu)$ and each $\alpha>
C\|f\|_{L^1(\mu)}/\mu(M)$, there exist a decomposition $f=g+\sum_{i}b_i$,
sets $\{Q_i\}_i\subset M$, positive numbers $\{r_i\}_{i}$ and points $\{x_i\}_i\subset M$ such that for all $i$,
\begin{enumerate}
\vspace{-0.25cm}
\item[(i)] $\supp b_i\subset Q_i$ and $\int_M b_i\,d\mu=0$,
\vspace{-0.25cm}
\item[(ii)] $\|g\|_{L^\fz(\mu)}\le C\alpha$ and $\sum_i\|b_i\|_{L^1(\mu)}
\le C\|f\|_{L^1(\mu)}$, \vspace{-0.25cm}
\item[(iii)] $Q_i\subset B(x_i, Cr_i)$ and $\sum_i\mu(Q_i^\ast)
\le C\|f\|_{L^1(\mu)}/\alpha$, where $Q_i^\ast\equiv\{x\in M:\, d(x,
Q_i)<r_i\}$; \vspace{-0.7cm}
\end{enumerate}
see \cite[Definition~1.1]{hs}. Hebisch and Steger
\cite[Theorem~1.2]{hs} further proved that, for any Calder\'on--Zygmund space $(M, d,\mu)$, if
$T\equiv\sum_{j\in\zz}K_j$ is bounded on $L^2(\mu)$ and if there
exist positive constants $c\in (0,1)$, $C$, $a$ and $b$ such that
\begin{equation}\label{1.3}
\dint_M|K_j(x, y)|(1+c^jd(x, y))^a\,d\mu(x)\le C
\qquad\forall\, y\in M,
\end{equation}
and
\begin{equation}\label{1.4}
\dint_M|K_j(x, y)-K_j(x, z)|\,d\mu(x)\le C(c^jd(y, z))^b
\qquad\forall\, y,z\in M,
\end{equation}
then $T$ is of weak type $1$ and bounded on $L^p(\mu)$ for all
$p\in(1, 2]$.

Obviously, spaces of homogeneous type in the sense of Coifman and
Weiss \cite{cw2} enjoy the Calder\'on--Zygmund property. An example
of Calder\'on--Zygmund spaces, which is not a space of homogeneous type, is $(S, d, \rho)$
(see \cite[Lemma~5.1]{hs}).
In fact, each integrable function $f$ on  $(S,\rho)$ at any level $\alpha>0$
can be decomposed into a sum $g+\sum_i b_i$ as above, where every
$b_i$ is supported on a set $R_i$ which belongs to a suitable family $\crz$
of sets in $S$: these sets are not balls because of the
exponential growth of the space, but suitable ``rectangles"
in $\mathbb{R}^n \ltimes \mathbb{R}^+ $. More precisely, for any $R\in\crz$, there exists a positive $r_R$ such that $R$ is contained in a ball of radius comparable to $r_R$ and
the measures of $R$ and its dilated set
$R^\ast\equiv\{x\in S:\, d(x, R)<r_R\}$ are comparable. The elements
in $\crz$ are called  {\it Calder\'on--Zygmund sets} on $(S, d,
\rho)$. As an application of the aforementioned Calder\'on--Zygmund
theory of singular integrals, spectral multipliers associated with a
distinguished Laplacian $\Delta$ on $(S, d, \rho)$ which satisfy certain Mihlin-H\"ormander type
condition are of weak type $1$ and bounded on $L^p$ for all $p\in(1, \fz)$ \cite[Theorem~2.4]{hs}.
M\"uller and Thiele \cite{mt} re-obtained the multiplier results of \cite{hs}
by considering estimates of the wave propagator associated with $\Delta$.

A reformulation of Hebisch--Steger \cite[Theorem~1.2]{hs} on $(S,d, \rho)$, using
a condition of H\"ormander's type, is as follows: if  $T$ is a linear operator which
is bounded on $L^2$ and admits a locally integrable kernel $K$
off the diagonal satisfying that
\begin{equation}\label{1.5}
\sup_{R\in\crz}\sup_{y,\, z\in R}\dint_{S\setminus R^\ast} |K(x,
y)-K(x, z)|\,d\rho(x)<\fz,
\end{equation}
then $T$ is bounded from $L^1$ to $L^{1,\, \fz}$ and on
$L^p$ for all $p\in(1, 2]$; see \cite{va2} for the details and see also \cite{st93,
g} for the Euclidean case.
For the endpoint case, Vallarino \cite{va} developed
an $H^1-\BMO$ theory on $(S, d, \rho)$, and proved that singular
integrals whose kernels satisfy the condition \eqref{1.5} are bounded from $H^1$ to
$L^1$ and from $L^\fz$ to $\BMO$. As an application,
spectral multipliers associated with a distinguished Laplacian $\Delta$ which satisfy certain
Mihlin-H\"ormander type condition are bounded from $H^1$ to
$L^1$ and from $L^\fz$ to $\BMO$ (see
\cite[Proposition~4.2]{va}). Moreover, Sj\"ogren and Vallarino \cite{sv} considered $H^1-L^1$
boundedness of various Riesz transforms associated with $\Delta$. In \cite{va3},
the Calder\'on--Zygmund theory of \cite{hs}  is generalized to Damek--Ricci spaces.

In this paper, we study the boundedness of maximal singular integrals on
$(S, d, \rho)$.
The importance of results in this direction is well known, and comes from
the fact that they imply pointwise convergence results (see,
for example, \cite{g} and, in particular, \cite[Theorem 2.1.14]{g}); see also Stein--Weiss \cite[p.\,60, Theorem 3.12]{sw71} or
Stein \cite[p.\,42, Theorem 4]{st70}.
Recall that in the
Euclidean setting, the {\em maximal singular integral} $T^\ast$ associated
with a kernel $K$ is  defined, for all suitable functions $f$ and all $x$ in $\rn$, by
$$T^\ast f(x)=\sup_{\ez>0}\bigg|\dint_{|x-y|>\ez}K(x, y)f(y)\,dy\bigg|.$$
An alternative but equivalent way of expressing this {\em operator}
$T^\ast$ is
$$T^\ast f(x)=\sup_{B\subset \rn,\, B\ni x}
\bigg|\dint_{\rn\setminus 2B}K(x, y)f(y)\,dy\bigg|,$$
where the supremum is taken over all Euclidean balls in $\rn$ containing $x$.

In view of this observation,  in the space $(S, d, \rho)$,
we define the \emph{maximal
singular integral} $T^\ast$ associated with a kernel $K$ as
\begin{equation}\label{1.6}
T^\ast f(x) \equiv \sup_{R\in\crz,\, R\ni x} \bigg|\dint_{S \setminus
R^\ast}K(x, y)f(y)\,d\rho(y)\bigg| \quad\quad\forall\, x\in S,
\end{equation}
where $f\in L_c^\infty$ and the supremum is taken over all
Calder\'on--Zygmund sets $R\in
\crz$ containing $x$; see Section \ref{s4} below for more details.

The main aim of this paper is to prove that, for $T^\ast$, defined as in
\eqref{1.6}, and associated with a kernel $K$ that satisfies an
integral size condition and H\"ormander's condition (see \eqref{4.1}
and \eqref{4.2} below), the following {\em four statements are equivalent}:
\begin{enumerate}
\vspace{-0.25cm}
\item[(i)] $T^\ast$ is bounded from $L_c^\fz$ to  $\BMO$;
\vspace{-0.25cm}
\item[(ii)] $T^\ast$ is bounded on $L^p$ for all $p\in(1, \fz)$;
\vspace{-0.25cm}
\item[(iii)] $T^\ast$ is bounded on $L^p$ for some $p\in(1, \infty)$;
\vspace{-0.25cm}
\item[(iv)] $T^\ast$ is bounded from $L^1$ to $L^{1,\fz}$;
\vspace{-0.25cm}
\end{enumerate}
see Theorem \ref{t4.1} below.
Moreover,  if $T$ is further assumed to be bounded on $L^2$, then the
above boundedness (i) through (iv) hold for $T^\ast$; see Theorem \ref{t4.2} below.

The proof of the main results of the paper, namely,
Theorems \ref{t4.1} and \ref{t4.2},
are presented in Section \ref{s4}.
The main ingredients used in the proof are the
Calder\'on--Zygmund property of $(S, d, \rho)$
and  certain Fefferman--Stein weak type inequalities
related to the local sharp maximal functions in the sense of John
\cite{john}, Str\"omberg \cite{stromberg} and
Jawerth--Torchinsky \cite{jt}; see Propositions \ref{p3.1} and
\ref{p3.2} below.
The proof of the aforementioned Fefferman--Stein type inequalities
relies on the existence of certain ``dyadic" sets on $(S, d, \rho)$,
which are an analogue of the Euclidean dyadic cubes and were constructed in \cite{lvy};
see Lemma \ref{dyadicgrid} below.
We remark that the proof
of Theorem \ref{t4.1} invokes some ideas of \cite{hyy} and Grafakos
\cite{gra}.

Some applications are given in Section \ref{s5}. Precisely,
for certain class of spectral multipliers for the distinguished Laplacian $\Delta$
we prove in Theorem \ref{t5.1} below that the corresponding maximal
singular integral operators are bounded from $L_c^\fz$ to $\BMO$,
from $L^1$ to $L^{1,\fz}$ and  on $L^p$ for all $p\in(1,
\fz)$. The main difficulty in proving Theorem \ref{t5.1} is to show that
the kernels of such spectral multipliers satisfy the integral size condition
\eqref{4.1}, which requires very delicate estimates (see Proposition \ref{p5.1} below).

Our paper is organized as follows. A brief recall of the geometric properties of $S$ and the Calder\'on--Zygmund property is presented in Section \ref{s2}.
In Section \ref{s3}, we establish the Fefferman--Stein (weak)
type inequalities related to
the local sharp maximal functions on $S$. The whole Section
\ref{s4} is devoted to the
proof of Theorems \ref{t4.1} and \ref{t4.2}, which are the main results of the paper.
Finally, an application to the spectral multipliers for the Laplacian
$\Delta$ is studied in Section \ref{s5}.

We make some conventions on notation. Let $\nn \equiv \{0, 1,\,2,\,\cdots\}$ and
$\rr^+ \equiv (0, \fz)$. For any space $X$ and any subset $E$  of $X$, set
$E^\complement \equiv X\setminus E$ and let $\chi_E$ denote the
{\em characteristic function} of $E$. Denote by $C$ a
{\em positive constant} independent of the main parameters involved, which
may vary at different occurrences. {\em Constants with subscripts} do not
change through the whole paper. We use $f\ls g$ to denote $f\le Cg$.
If $f\ls g\ls f$, we write $f\approx g$.
For an operator $T$ defined on the Banach space $\mathcal A$ and taking values in
the Banach space $\mathcal B$, we use $\|T\|_{\mathcal A\to \mathcal B}$ to denote
its {\em operator norm}.

\section{Preliminaries}\label{s2}

We recall the notion of Calder\'on--Zygmund sets,
which appears in \cite{hs} and implicitly in \cite{gs}.
Let $\cq$ be the {\em collection of dyadic cubes in $\rn$}.

\begin{defn}\rm\label{d2.1}
A \emph{Calder\'on--Zygmund set} is a set $R=Q\times[ae^{-r}, ae^r)$, where
$Q\in\cq$ with side length $L$, $a\in\rr^+$, $r>0$ and
\begin{eqnarray*}
&&e^2ar\le L<e^8ar \,\,\, \mbox{ if }\,\,\, r<1;  \quad\quad\quad
ae^{2r}\le L<ae^{8r}\,\,\, \mbox{ if }\,\,\, r\ge1.
\end{eqnarray*}
Set $a_R \equiv  a$, $r_R \equiv  r$ and $x_R \equiv (c_Q, a)$, where
$c_Q$ is the \emph{center} of $Q$. Denote by $\crz$ the \emph{family of all
Calder\'on--Zygmund sets on $S$}.
For any  $x\in S$, let $\crz(x)$  be the \emph{collection of
all $R\in\crz$ containing $x$}.
\end{defn}

The following lemma presents some properties of the Calder\'on--Zygmund sets
(see \cite{hs, va}).

\begin{lem}\label{l2.1}
There exists $\kz_0\in[1, \fz)$ such that for all $R\in\crz$, the
following hold:
\begin{enumerate}
\vspace{-0.2cm}
\item[\rm(i)] $B(x_R, r_R)\subset R\subset B(x_R, \kz_0r_R)$;
\vspace{-0.25cm}
\item[\rm(ii)] $\rho(R^\ast)\le\kz_0\rho(R)$, where
$R^\ast\equiv\{x\in S:\, d(x, R)<r_R\}$ is the dilated set of
$R\in\crz$;
\vspace{-0.25cm}
\item[\rm (iii)] every $R\in\crz$ can be decomposed into
mutually disjoint sets $\{R_i\}_{i=1}^k\subset\crz$ with $k=2$ or
$k=2^n$ such that $R=\cup_{i=1}^k R_i$ and
$\rho(R_i)=\rho(R)/k$ for all $i\in\{1,\cdots, k\}$.
\end{enumerate}
\end{lem}

Using the geometric properties of the Calder\'on--Zygmund sets,
the authors, in \cite{lvy}, constructed certain ``dyadic" sets on $(S, d, \rho)$,
which are analogues of the Euclidean.

\begin{lem}\label{dyadicgrid}
There exists a sequence $\{\cp_j\}_{j\in\zz}$ such that each $\cp_j$
consists of pairwise disjoint Calder\'on--Zygmund sets, and
\begin{enumerate}
\vspace{-0.25cm}
\item[\rm (i)] for any given $j\in\zz$,  $S=\cup_{R\in\cp_j} R$;
\vspace{-0.25cm}
\item[\rm (ii)] if $\ell\le k$, $R\in\cp_\ell$ and $R^\pz\in\cp_k$, then
either $R\subset R^\pz$ or $R\cap R^\pz=\emptyset$; \vspace{-0.25cm}
\item[\rm (iii)] for any given $j\in\zz$ and $R\in\cp_j$, there
exists a unique $R^\pz\in\cp_{j+1}$ such that $R\subset R^\pz$ and
$\rho(R^\pz)\le \max\{2^n, 3\}\rho(R)$; \vspace{-0.25cm}
\item[\rm (iv)] for any $j\in\zz$, every $R\in\cp_j$ can be decomposed into
mutually disjoint sets $\{R_i\}_{i=1}^k\subset\cp_{j-1}$ with $k=2$
or $k=2^n$ such that $R=\cup_{i=1}^k R_i$ and
$\rho(R)/2^n \le\rho(R_i)\le 2\rho(R)/3$ for all $i\in\{1,\cdots, k\}$.
\vspace{-0.25cm}
\end{enumerate}
From now on, we set $\cp \equiv \{\cp_j\}_{j\in\zz}$.
\end{lem}

Hardy--Littlewood maximal functions on groups with exponential growth have been
investigated in a series of works; see, for example, \cite{gghm, gs, va1}.
For any $f\in L_\loc^1$, we define the {\em Hardy--Littlewood
maximal function} $\cm f$
and the {{\em dyadic  Hardy--Littlewood maximal function}} $\cm_\cp f$
respectively
by the formulae
\begin{equation}\label{2.1}
\cm f(x) \equiv \sup_{R\in\crz(x)}\dfrac1{\rho(R)}\dint_R|f|\,d\rho
\quad\forall\  x\in S,
\end{equation}
and
\begin{equation}\label{2.2}
\cm_\cp f(x) \equiv \sup_{R\in\crz(x),\,R\in\cp}\dfrac1{\rho(R)}\dint_R|f|\,d\rho
\quad\forall\  x\in S.
\end{equation}
The maximal function $\mathcal M$ has the following boundedness properties \cite{va1}.
\begin{prop}\label{p2.1}
 $\cm$ is bounded from $L^1$ to $L^{1,\fz}$, and on $L^p$ for
all $p\in(1, \fz]$.
\end{prop}
From
\eqref{2.1}, \eqref{2.2} and the differentiation theorem for
integrals, it follows that
$$f(x)\le\cm_\cp f(x)\le\cm f(x) \quad\quad
\forall\  f\in L_\loc^1,\,\mbox{ almost  every}\  x\in S.$$
Thus,  the operator $\cm_\cp$
is bounded from $L^1$ to $L^{1,\infty}$ and for any $p\in(1, \infty]$,
$$\|\cm_\cp(f)\|_{L^p}\approx \|f\|_{L^p} \approx
\|\cm(f)\|_{L^p} \quad \forall\, f\in L^p. $$

It was proved in \cite[Lemma~5.1]{hs} that $(S, d, \rho)$ possesses
the Calder\'on--Zygmund property. Indeed, the boundedness
properties of $\cm_\cp$ and the differentiation theorem for integrals,
together with Lemma \ref{dyadicgrid}  and a
standard stopping-time argument, imply the following dyadic version of
the Calder\'on--Zygmund property; see also \cite[Proposition~2.4]{va}.

\begin{prop}\label{p2.2}
Let $f\in L^p$ with $p\in[1, \fz)$. For any $\az>0$,
there exists a sequence of disjoint sets, $\{R_i\}_i\in\cp$, such
that $f$ can be decomposed into $f=g+\sum_i b_i$, where
\begin{enumerate}
\vspace{-0.25cm}
\item[\rm (i)] $|g(x)|\le C_1\az$ for almost every $x\in S$;
\vspace{-0.25cm}
\item[\rm (ii)] for all $i$, $\supp b_i\subset R_i\in\cp$ and $\int_S b_i\,d\rho=0$;
\vspace{-0.25cm}
\item[\rm (iii)] for all $i$, $\az\le\{\frac1{\rho(R_i)}
\int_{R_i}|f|^p\,d\rho\}^{1/p}\le C_1\az$;
\vspace{-0.25cm}
\item[\rm (iv)] for all $i$, $\|b_i\|_{L^p}\le C_1\az[\rho(R_i)]^{1/p}$,
\vspace{-0.25cm}
\end{enumerate}
where $C_1$ is a positive constant independent of $\az$ and $f$.
\end{prop}

\begin{rem}\rm\label{r2.1}
If $f\in L^\fz\cap L^p$ for some $p\in [1,\fz)$ and $f=g+\sum_i b_i$ is a
Calder\'on--Zygmund decomposition of $f$ obtained as in Proposition \ref{p2.2},
then $\|g\|_{L^\fz}\le \|f\|_{L^\fz}$.
Moreover, if $\int_Sf\,d\rho=0$, then $\int_S g\,d\rho=0$.
 \end{rem}

\section{Fefferman--Stein Type Inequalities}\label{s3}

The local maximal functions in  Euclidean
spaces were introduced by John \cite{john} and later investigated by
Str\"omberg \cite{stromberg}, Jawerth--Torchinsky \cite{jt} and Lerner
\cite{lerner}; see also \cite{hyy} for
the setting of spaces of homogeneous type.
Following this pioneering work, we introduce the local
maximal functions on $(S, d, \rho)$.

\begin{defn}\rm\label{d3.1}
(i) The \emph{non-increasing rearrangement} of a measurable function $f$ on
$(S, d, \rho)$ is defined by
$$f^\ast(\xi)  \equiv  \inf\{t>0:\, \rho(\{x\in S:\, |f(x)|>t\})<\xi\}
\quad\forall\ \xi\in(0, \fz).$$

(ii) Let $s\in(0, 1)$ and $\cp$ be the family of dyadic sets. For any
$f\in L_\loc^1$, its \emph{local maximal function}
$\cm_{0, s} f$ is defined by
$$\cm_{0, s} f(x) \equiv
\sup_{R\in\crz(x)} \lf(f\chi_R\r)^\ast\big(s\rho(R)\big)
\quad\forall\ x\in S,$$
and its \emph{dyadic local maximal function} $\cm_{0, s}^\cp f$ is defined by
$$\cm_{0, s}^{\cp} f(x) \equiv
\sup_{R\in\crz(x)\cap \cp}\lf(f\chi_R\r)^\ast\big(s\rho(R)\big)
\quad\forall\ x\in S.$$
\end{defn}

Some properties of these local maximal operators are presented in
the following lemma.

\begin{lem}\label{l3.1}
 For any $f\in L_\loc^1$, the following hold:
\begin{enumerate}
\vspace{-0.25cm}
\item[\rm(i)] for all $s\in(0, 1)$ and $x\in S$,
$\cm_{0, s}^{\cp} f(x)\le \cm_{0, s} f(x)$;
\vspace{-0.25cm}
\item[\rm(ii)] if $0<s_1<s_2<1$, then $\cm_{0, s_2}^\cp f(x)
\le \cm_{0, s_1}^\cp f(x)$ for all $x\in S$;
\vspace{-0.25cm}
\item[\rm(iii)] for all $s\in(0, 1)$ and $x\in S$,
$\cm_{0, s}^{\cp} f(x)\le s^{-1}\cm_\cp f(x)$;
\vspace{-0.25cm}
\item[\rm(iv)] for all $s\in(0, 1)$ and all
measurable functions $f_1$ and $f_2$,
$$\cm_{0, s}^{\cp}(f_1+f_2)(x)\le \cm_{0, s/2}^\cp(f_1)(x)+\cm_{0, s/2}^\cp (f_2)(x)
\qquad\forall\, x\in S.$$
\vspace{-0.8cm}
\item[\rm(v)] Properties (ii)--(iv) hold with
$\cm_{0, s}^{\cp}$ and $\cm_\cp$ replaced by $ \cm_{0, s}$ and $\cm$,
respectively.
\vspace{-0.25cm}
\item[\rm(vi)] For all $s\in(0, 1)$, $\lz\in(0, \fz)$ and $f\in L_\loc^1$,
\begin{eqnarray*}
\{x\in S:\, \cm_\cp(\chi_{\{|f|>\lz\}})(x)>s\}
&&\subset \{x\in S:\, \cm_{0, s}^{\cp} f(x)>\lz\}\\
&&\subset \{x\in S:\, \cm_\cp(\chi_{\{|f|>\lz\}})(x)\ge s\}.
\vspace{-0.25cm}
\end{eqnarray*}
Consequently,
\begin{eqnarray*}
\rho(\{x\in S:\, |f(x)|>\lz\})
&&\le \rho(\{x\in S:\, \cm_{0, s}^{\cp} f(x)>\lz\})\\
&&\le \|\cm_\cp\|_{L^1\to L^{1,\fz}} s^{-1} \rho(\{x\in S:\,
|f(x)|>\lz\}).
\end{eqnarray*}
\vspace{-0.8cm}
\item[\rm(vii)] Similarly,
\begin{eqnarray*}
\{x\in S:\, \cm(\chi_{\{|f|>\lz\}})(x)>s\}
&&\subset \{x\in S:\, \cm_{0, s} f(x)>\lz\}\\
&&\subset \{x\in S:\, \cm(\chi_{\{|f|>\lz\}})(x)\ge s\},
\vspace{-0.25cm}
\end{eqnarray*}
and
\begin{eqnarray*}
\rho(\{x\in S:\, |f(x)|>\lz\})
&&\le \rho(\{x\in S:\, \cm_{0, s} f(x)>\lz\})\\
&&\le \|\cm\|_{L^1\to L^{1,\fz}} s^{-1} \rho(\{x\in S:\,
|f(x)|>\lz\}).
\end{eqnarray*}
\vspace{-0.5cm}
\end{enumerate}
\end{lem}

\begin{proof}
Properties (i) and (ii) follow from the
definition of $\cm_{0, s}^{\cp}$. For all $R\in \crz(x)\cap \cp$,
\begin{eqnarray*}
\rho(\{y\in R:\, |f(y)|> s^{-1}\cm_\cp f(x)\}) &&< \int_R\frac{|f(y)|
}{s^{-1}\cm_\cp f(x)}\,d\rho(y)\le s\rho(R),
\end{eqnarray*}
therefore, (iii) holds. Property (iv) follows from an argument similar
to the one used in \cite[Lemmas~2.2]{hyy}, while (v)
is proven as (i)-(iv). Finally, proceeding as in the  proof of
\cite[Lemmas~2.3]{hyy} yields (vi) and (vii).
\end{proof}

Next, we recall the notion of the median value;
see  \cite{stromberg} for the
Euclidean setting.

\begin{defn}\rm\label{d3.2}
Suppose that $f$ is a real function in $L^1_{loc}$ and
$R\in\crz$. A \emph{median value} $m_f(R)$ of $f$ over $R$ is defined to
be one of the real numbers satisfying
$$\rho(\{x\in R:\, f(x)>m_f(R)\})\le\rho(R)/2,$$
and
$\rho(\{x\in R:\, f(x)<m_f(R)\})\le\rho(R)/2.$
In the case when $f$ is complex, define
$$m_f(R)\equiv  m_{\Re(f)}(R)+i m_{\Im(f)}(R),$$
where $\Re(f)$ and $\Im (f)$ denote, respectively, the \emph{real part}
and the \emph{imaginary part} of $f$.
\end{defn}

In the following lemma we show an analogue of the inequality proved by
Jawerth and  Torchinsky in \cite[p.\,238]{jt}, which
will be used in the proof of ``good-$\lz$" inequalities in Lemma \ref{l3.5} below.

\begin{lem}\label{l3.2}
Let $f\in L_\loc^1$ and $R\in\crz$. Then
\begin{equation}\label{3.1}
|m_{f}(R)|\le\sqrt 2\inf\big\{t>0:\, \rho(\{y\in R:\,
|f(y)|>t\})<\rho(R)/2\big\}.
\end{equation}
Consequently,  for all $s\in(0,1/2]$ and $R\in\cp$,
\begin{equation}\label{3.2}
|m_{f}(R)|\le\sqrt 2\inf_{x\in R}\cm_{0, 1/2}^\cp f(x) \le \sqrt
2\inf_{x\in R}\cm_{0, s}^{\cp} f(x).
\end{equation}
\end{lem}

\begin{proof}
{We first} show \eqref{3.1}. If
$f$ is real and $m_f(R)>0$, then for all $\ez\in\big(0, m_f(R)\big)$,
$$\rho(\{y\in R:\, |f(y)|> m_{f}(R)-\ez\})
\ge\rho(\{y\in R:\, f(y)\ge m_{f}(R)\})\ge\rho(R)/2,$$
which implies that
$$\inf\big\{t>0:\, \rho(\{y\in R:\, |f(y)|>t\})<\rho(R)/2\big\}\ge m_{f}(R)-\ez.$$
Letting $\ez\to0$ yields
\begin{equation}\label{3.3}
\inf\big\{t>0:\, \rho(\{y\in R:\, |f(y)|>t\})<\rho(R)/2\big\}\ge m_{f}(R).
\end{equation}
If $f$ is real and $m_f(R)<0$, applying the above argument to
$-f$ and $-m_f(R)$, we also obtain \eqref{3.3}.

For any complex
function $f$, we have
$$|m_f(R)|^2=|m_{\Re f}(R)|^2+|m_{\Im f}(R)|^2
\le 2\max\{|m_{\Re f}(R)|^2,\, |m_{\Im f}(R)|^2\}.$$
Without loss of generality, we may assume that $|m_{\Re f}(R)|\le
|m_{\Im f}(R)|$. Then, by \eqref{3.3},
\begin{eqnarray*}
|m_f(R)|\le\sqrt 2|m_{\Im f}(R)|
&&\le \sqrt 2 \inf\big\{t>0:\, \rho(\{y\in R:\, |\Im f(y)|>t\})<\rho(R)/2\big\}\\
&&\le \sqrt 2 \inf\big\{t>0:\, \rho(\{y\in R:\, |f(y)|>t\})<\rho(R)/2\big\}.
\end{eqnarray*}
This proves \eqref{3.1}. The inequalities \eqref{3.2} easily follow
from the definition of $\cm_{0, 1/2}^\cp  $ and Lemma \ref{l3.1}(ii).
\end{proof}

Vallarino \cite{va} introduced the {\em space} $\BMO$ of functions
with bounded mean oscillation on $(S, d, \rho)$  as follows.
For any $f\in L_\loc^1$ and $R\in\crz$, set
$f_R \equiv \frac1{\rho(R)}\int_R f\,d\rho$;
the function
$f$ is said to be in $\BMO$ if
\begin{equation*}
 \|f\|_{\BMO} \equiv \sup_{R\in\crz}
\dfrac1{\rho(R)}\dint_R|f-f_R|\,d\rho<\fz.
\end{equation*}
For any $q\in(0, \fz)$ and $f\in L_\loc^1$, set
$$ \|f\|_{\BMO_q} \equiv \sup_{R\in\crz}
\lf\{\dfrac1{\rho(R)}\dint_R|f-f_R|^q\,d\rho\r\}^{1/q}.$$
It was proved in \cite[Section~3]{va} that for any $q\in(1, \fz)$,
there exists a positive constant $C_q$ such that for all $f\in L_\loc^1$,
\begin{equation}\label{3.4}
\dfrac1{C_q}\|f\|_{\BMO_q}\le \|f\|_{\BMO}\le C_q\|f\|_{\BMO_q}.
\end{equation}
It turns out that \eqref{3.4} also holds for $q\in(0, 1)$.
This is proved in the following lemma,
using some ideas of \cite{ly84}.

\begin{lem}\label{l3.3}
For any $\sz\in(0, 1)$ there exists a positive constant $C_\sz$,
which depends only on $\sz$, such that for all $f\in L_\loc^1$,
\begin{equation}\label{3.5}
 C_\sz\|f\|_{\BMO}\le \|f\|_{\ast,\,\sz}
 \le \|f\|_{\BMO_\sz} \le\|f\|_{\BMO},
\end{equation}
where
$$\|f\|_{\ast,\,\sz} \equiv \sup_{R\in\crz}\inf_{c\in\cc}
 \lf\{\dfrac1{\rho(R)}\dint_R|f(x)-c|^\sz\,d\rho(x)\r\}^{1/\sz}.$$
\end{lem}

\begin{proof}
The third inequality of \eqref{3.5} follows from H\"older's
inequality, while the second inequality of \eqref{3.5} follows
from the definitions of $\|\cdot\|_{\ast,\,\sz}$ and
$\|\cdot\|_{\BMO_\sz}$.

To prove the first inequality of \eqref{3.5}, we fix $f\in L_\loc^1$. Suppose that
$\|f\|_{\ast,\,\sz}<\fz$; otherwise there is nothing to prove. For
any $R\in\crz$, by the  local integrability of $f$ and the
dominated convergence theorem, we have that
$\{\frac1{\rho(R)}\int_R|f-c|^\sz\,d\rho\}^{1/\sz}$
is continuous with respect to $c$ and  it tends to infinity as
$|c|\to\fz$. Thus, for any fixed $R\in\crz$, there exists a certain
$A_R\in\cc$ such that
\begin{eqnarray}\label{3.6}
\inf_{c\in\cc}
\lf\{\dfrac1{\rho(R)}\dint_R|f-c|^\sz\,d\rho\r\}^{1/\sz}
&=&\lf\{\dfrac1{\rho(R)}\dint_R|f-A_R|^\sz\,d\rho\r\}^{1/\sz}\le\|f\|_{\ast,\,\sz}.
\end{eqnarray}
For all $R'\in\crz$, by \eqref{3.6} and the fact that
$|a^\sz-b^\sz|\le|a-b|^\sz$ for all $a,\, b\in(0, \fz)$, we have
$$\dfrac1{\rho (R')}\dint_{R'}\Big||f-A_R|^\sz
-|A_R-A_{R'}|^\sz\Big|\,d\rho \le \dfrac1{\rho
(R')}\dint_{R'}|f-A_{R'}|^\sz\,d\rho
\le\|f\|_{\ast,\,\sz}^\sz,$$ which implies that
\begin{equation}\label{3.7}
\big\||f-A_R|^\sz\big\|_{\BMO}\le 2\|f\|_{\ast,\,\sz}^\sz.
\end{equation}
For all $a, b\in(0, \fz)$, we have
$(a+b)^{1/\sz}\le2^{1/\sz-1}\lf[a^{1/\sz}+b^{1/\sz}\r]$, which
together with \eqref{3.6} yields that for all $R\in\crz$,
\begin{eqnarray}\label{3.8}
&&\dfrac1{\rho(R)}\dint_R|f-A_R|\,d\rho=\dfrac1{\rho(R)}
\dint_R\lf\{|f-A_R|^\sz\r\}^{1/\sz}\,d\rho(x)\noz\\
&&\hs\le 2^{1/\sz-1}\lf\{\dfrac1{\rho(R)}
\dint_R\lf||f(x)-A_R|^\sz-\dfrac1{\rho(R)}
\dint_R|f(y)-A_R|^\sz\,d\rho(y)\r|^{1/\sz}\,d\rho(x)\r.\noz\\
&&\hs\quad+\lf. \lf[\dfrac1{\rho(R)}
\dint_R|f(y)-A_R|^\sz\,d\rho(y)\r]^{1/\sz}\r\}\noz\\
&&\hs\le
2^{1/\sz-1}\lf\{\lf(\lf\||f-A_R|^\sz\r\|_{\BMO_{1/\sz}}\r)^{1/\sz}
+\|f\|_{\ast,\,\sz}\r\}.\noz
\end{eqnarray}
Observe that \eqref{3.4} and \eqref{3.7} imply that
$$\lf(\lf\||f-A_R|^\sz\r\|_{\BMO_{1/\sz}}\r)^{1/\sz}
\le \lf(C_{1/\sz}\||f-A_R|^\sz\|_{\BMO}\r)^{1/\sz}
\le\lf(2C_{1/\sz}\r)^{1/\sz} \|f\|_{\ast,\, \sz},$$ where
$C_{1/\sz}$ is as in \eqref{3.4}.
Applying this estimate, we see that
$$\frac1{\rho(R)}\int_R|f-f_R|\,d\rho
\le\frac2{\rho(R)}\int_R|f-A_R|\,d\rho
\le 2^{1/\sz} [(2C_{1/\sz})^{1/\sz}+1] \|f\|_{\ast,\, \sz}.$$
By taking the supremum over all
$R\in\crz$ we obtain
$\|f\|_{\BMO}\ls \|f\|_{\ast,\,\sz}$, which completes the proof.
\end{proof}

Now we introduce the (local) sharp maximal functions on $(S, d,
\rho)$; see \cite{g, st93, stromberg} for their definitions in the
Euclidean setting and \cite{hyy} for their definitions
in spaces of homogeneous type.

\begin{defn}\rm\label{d3.3}
(i) For any  $f\in L_\loc^1$,
its \emph{sharp maximal function} $f^\sharp$ is defined by
$$f^\sharp(x) \equiv \sup_{R\in\crz(x)}
\dfrac1{\rho(R)}\dint_R|f(y)-f_R|\,d\rho(y) \quad\forall\  x\in S.$$
(ii) Let $s\in(0, 1)$. For any $f\in L_\loc^1$,
its \emph{local sharp maximal function} $\cm_{0, s}^\sharp f$ is defined by
$$\cm_{0, s}^\# f (x) \equiv
\sup_{R\in\crz(x)}
\inf_{c\in\cc}\big((f-c)\chi_R\big)^\ast\big(s\rho(R)\big)
\quad\forall\  x\in S.$$
\end{defn}

The following lemma summarizes some properties of the (local)
sharp maximal functions on $(S, d, \rho)$. The proofs are easy and
hence omitted; see \cite{g, stromberg, hyy, jt}.

\begin{lem}\label{l3.4}
Let $f,\,g\in L_\loc^1$. Then, for all $x\in S$,
\begin{enumerate}
\vspace{-0.25cm}
\item[\rm(i)]  $f^\sharp(x)\le 2\cm f(x)$, where $\cm$ is defined in \eqref{2.1};
\vspace{-0.25cm}
\item[\rm(ii)]  $|f|^\sharp(x)\le 2f^\sharp(x)$ and
$(f+g)^\sharp(x)\le f^\sharp(x)+g^\sharp(x)$;
\vspace{-0.25cm}
\item[\rm(iii)] $f^\sharp(x)/2\le\sup_{R\in\crz(x)}\inf_{a\in\cc}\frac1{\rho(R)}
\int_R|f(y)-a|\,d\rho(y)\le f^\sharp(x)$; \vspace{-0.25cm}
\item[\rm(iv)] when $s\in(0, 1)$, $\cm_{0, s}^\# f (x)\le \cm_{0, s} f(x)$ and
 $\cm_{0, s}^\# f (x)\le s^{-1}f^\#(x)$, where $\cm_{0, s}$ is defined in \eqref{d3.1};
\vspace{-0.25cm}
\item[\rm(v)] if $s\in(0,  1/2]$ and $R\in\crz$, then
$$\rho(\{y\in R:\, |f(y)-m_f(R)|>2\sqrt2 \inf_{x\in R}\cm_{0, s}^\# f
(x)\})\le s\rho(R).$$
\end{enumerate}
\end{lem}

For any  $p\in(1, \fz)$ and $s\in(0, 1)$, an immediate
consequence of Lemma \ref{l3.4} (i) and (iv) and
the $L^p$-boundedness of $\cm$  is that for all $f\in L^p$,
\begin{equation}\label{3.9}
\|\cm_{0, s}^\# f \|_{L^p}\le 2 s^{-1}\|\cm\|_{L^p\to
L^p}\|f\|_{L^p};
\end{equation}
Lemma \ref{l3.4}(iv) and Lemma \ref{l3.1}(vii) also
imply that for all  $f\in L^{p,\fz}$,
\begin{equation}\label{3.10}
\|\cm_{0, s}^\# f \|_{L^{p,\fz}}\le \|\cm\|_{L^1\to
L^{1,\fz}}^{1/p} \, s^{-1/p}\|f\|_{L^{p,\fz}}.
\end{equation}
Indeed, the converses of \eqref{3.9} and \eqref{3.10} hold for
small $s$. To see this, we need the following  certain kind of {\em ``good-$\lz$" inequalities}
involving the dyadic local maximal function and the local sharp maximal function.

\begin{lem}\label{l3.5}
Let $0<s_1, s_2\le \frac12$. Then, there exists a positive constant
$C_2$ depending only on $S$ such that for all  $f\in L_\loc^1$ and all $\lz>0$,
\begin{eqnarray*}
&&\rho(\{x\in S:\, \cm_{0, s_1}^\cp f(x)>3\lz, \,\cm_{0, s_2}^\# f (x)\le\lz/4\})\\
&&\quad\le C_2 s_1^{-1}s_2\rho\lf(\lf\{x\in S:\,\cm_{0,s_1}^\cp f(x)>\lz\r\}\r).
\end{eqnarray*}
\end{lem}

\begin{proof}
Take $\lambda >0$ and set $\boz_\lz \equiv \lf\{x\in S:\,\cm_{0, s_1}^\cp f(x)>\lz\r\}$. We may
assume that $\rho(\boz_\lz)<\fz$, otherwise there is nothing to
prove. For any fixed $x\in\boz_\lz$, there exists a ``maximal dyadic
cube" $R_x\in\cp$ containing $x$ such that
\begin{equation}\label{3.11}
\inf\big\{t>0:\, \rho(\{y\in R_x:\, |f(y)|>t\})<s_1\rho(R_x)\big\}>\lz.
\end{equation}
Here ``maximal dyadic cube" means that if $R'\in\cp\cap \mathcal R(x)$
satisfies \eqref{3.11},
 then $R'\subset R_x$. Such a ``maximal
dyadic cube" exists since $R_x\subset\boz_\lz$ and
$\rho(\boz_\lz)<\fz$. Let $\{R_j\}_{j\in I}$ be the collection of
all such ``maximal dyadic cubes" obtained by running $x$ over $\boz_\lz$.
From the maximality, it follows that any two ``maximal dyadic
cubes" are disjoint. Moreover, $\boz_\lz=\cup_{j\in I} R_j$.
Therefore,  to show Lemma \ref{l3.5}, it suffices to
prove that there exists a positive constant $C_2$ such that for all
$j\in I$,
\begin{equation}\label{3.12}
\rho\lf(\lf\{x\in R_j:\, \cm_{0, s_1}^\cp f(x)>3\lz,\, \cm_{0, s_2}^\# f
(x)\le\lz/4\r\}\r) \le C_2 s_1^{-1}s_2 \rho\lf(R_j\r).
\end{equation}

Fix $j\in I$. We may assume that there exists $x_0\in R_j$ such that
$\cm_{0, s_2}^\# f (x_0)\le\lz/4$; otherwise \eqref{3.12} holds
trivially. Suppose that $R_j\in\cp_{j_0}$ for some $j_0\in\zz$.
Using Lemma \ref{dyadicgrid}(iii), we take $\wz R_j$ to be the unique
Calder\'on--Zygmund set in $\cp_{j_0+1}$ that contains $R_j$. Then,
$\rho(\wz R_j)\le \max\{2^n, 3\}\rho(R_j)$. By the maximality of
$R_j$ and \eqref{3.11}, we have
\begin{equation}\label{3.13}
\inf\lf\{t>0:\, \rho(\{y\in \wz R_j:\, |f(y)|>t\})<s_1\rho(\wz
R_j)\r\}\le \lz.
\end{equation}
From this and Lemma \ref{l3.2} together with the hypothesis $s_1\le 1/2$,
it follows that
$$|m_{f}(\wz R_j)|
\le\sqrt 2\inf\{t>0:\, \rho(\{y\in \wz R_j:\, |f(y)|>t\})<\rho(\wz
R_j)/2\} \le \sqrt 2 \lz.$$ Thus,
$$\cm_{0, s_1/2}^\cp (m_{f}(\wz R_j)\chi_{\wz R_j})(x)\le |m_f(\wz R_j)|\le\sqrt 2\lz.$$

If $\cm_{0, s_1}^\cp (f)(x)>3\lz$, then by \eqref{3.13} and the definition of $\wz R_j$,
we have
$$\cm_{0, s_1}^\cp (f\chi_{\wz R_j})(x)=\cm_{0, s_1}^\cp (f)(x)>3\lz.$$
So applying Lemma \ref{l3.1}(iv) yields that for all $x\in R_j$
satisfying $\cm_{0, s_1}^\cp (f)(x)>3\lz$,
\begin{eqnarray}\label{triangineq}
&&\cm_{0, s_1/2}^\cp \lf((f-m_{f}(\wz R_j))\chi_{\wz R_j}\r)(x)\noz\\
&&\quad\ge
\cm_{0, s_1}^\cp \lf(f\chi_{\wz R_j}\r)(x)
-\cm_{0, s_1/2}^\cp \lf(m_{f}(\wz R_j)\chi_{\wz R_j}\r)(x)>\lz.
\end{eqnarray}
By using \eqref{triangineq}, Lemma \ref{l3.1}(vi),  $\cm_{0, s_2}^\# f (x_0)\le\lz/4$ and
Lemma \ref{l3.4}(v), we  obtain
\begin{eqnarray*}
&&\rho\lf(\lf\{x\in R_j:\, \cm_{0, s_1}^\cp f(x)>3\lz,\,
\cm_{0, s_2}^\# f (x)\le\lz/4\r\}\r)\\
&&\hs \le \rho\lf(\lf\{x\in R_j:\, \cm_{0, s_1/2}^\cp \lf((f-m_{f}(\wz R_j))\chi_{\wz R_j}\r)(x)>\lz\r\}\r)\\
&&\hs\le 2\|\cm\|_{L^1\to L^{1,\fz}} s_1^{-1}
\rho\lf(\lf\{x\in \wz R_j:\, |f(x)-m_{f}(\wz R_j)|>\lz\r\}\r)\\
&&\hs\le 2\|\cm\|_{L^1\to L^{1,\fz}} s_1^{-1}
\rho\bigg(\bigg\{x\in \wz R_j:\, |f(x)-m_{f}(\wz R_j)|>2\sqrt 2
\inf_{z\in\wz R_j} \cm_{0, s_2}^\# f (z)\bigg\}\bigg)\\
&&\hs \le 2\|\cm\|_{L^1\to L^{1,\fz}} s_1^{-1} s_2\rho(\wz
R_j).
\end{eqnarray*}
Hence, \eqref{3.12} holds with $C_2\equiv 2\max\{2^n, 3\} \|\cm\|_{L^1\to
L^{1,\fz}}$. This finishes the proof.
\end{proof}

Applying the previous``good-$\lz$" inequality, we prove the following Fefferman--Stein type inequality.

\begin{prop}\label{p3.1}
Let $p_0\in(0, \fz)$ and $C_2$ be as in Lemma \ref{l3.5}. Then, for
any $p\in(p_0, \fz)$ and $s\in(0, 1/2]$ satisfying $s< (2^23^p
C_2)^{-1} $, there exists a positive constant $C$ such that
\begin{equation*}
\|f\|_{L^p}\le C\|\cm_{0, s}^\# f \|_{L^p} \quad\forall\ f\in L^{p_0,\fz}.
\end{equation*}
\end{prop}

\begin{proof} Fix $f\in L^{p_0,\fz}$.
For any $N\in\nn$, set
$${\rm I_N} \equiv \int_0^{3N}\rho\lf(\lf\{x\in S:\,\cm_{0, 1/2}^\cp  f(x)>\lz\r\}\r)\,p\lz^{p-1}d\lz.$$ Then, Lemma \ref{l3.1}(vi)
and  $p> p_0$ imply that $\rm I_N<\fz$.
Applying Lemma \ref{l3.5} yields that
\begin{eqnarray*}
\rm I_N&&= 3^p\dint_0^{N}p\lz^{p-1}\rho\lf(\lf\{x\in S:\,
\cm_{0,1/2}^\cp f(x)>3\lz\r\}\r)\,d\lz\\
&&\le3^p\lf[\dint_0^{N}p\lz^{p-1} \rho\lf(\lf\{x\in S:\,\cm_{0,1/2}^\cp
f(x)>3\lz,
\cm_{0, s}^\sharp f (x)\le\lz/4\r\}\r)\,d\lz\r.\\
&&\quad+\lf.\dint_0^{N}p\lz^{p-1}
\rho\lf(\lf\{x\in S:\,\cm_{0, s}^\sharp f(x)>\lz/4\r\}\r)\,d\lz\r]\\
&&\le 3^p 2 C_2 s \dint_0^{N}p\lz^{p-1}\rho \lf(\lf\{x\in
S:\,\cm_{0, 1/2}^\cp  f(x)>\lz\r\}\r)\,d\lz
+3^p4^p \|\cm_{0, s}^\sharp f\|_{L^p}^p\\
&&\le  s 3^p 2 C_2 {\rm I}_N +3^p4^p \|\cm_{0, s}^\sharp f\|_{L^p}^p.
\end{eqnarray*}
Since $s< (3^p 2^2 C_2)^{-1}$,
we have ${\rm I}_N\ls \|\cm_{0, s}^\sharp f \|_{L^p}^p$.
Letting $N\to\fz$ and using  Lemma \ref{l3.1}(vi),
we obtain
$\|f\|_{L^p}\le\|\cm_{0, 1/2}^\cp  f\|_{L^p} \ls\|\cm_{0, s}^\# f
\|_{L^p},$ which completes the proof.
\end{proof}
The corresponding weak-type Fefferman--Stein inequality  is the following.

\begin{prop}\label{p3.2}
Let $p_0 \in (0, \fz)$ and $C_2$ be as in Lemma \ref{l3.5}. Then, for
any $p\in[p_0, \fz)$ and $s\in(0, 1/2]$ satisfying $s< (2^23^p
C_2)^{-1} $, there exists a positive constant $C$ such that
$$\|f\|_{L^{p,\,\fz}}\le C\|\cm_{0, s}^\sharp f\|_{L^{p,\,\fz}}
\quad\quad \forall\ f\in L^{p_0,\fz}.$$
\end{prop}

\begin{proof}
For any $N\in\nn$, set $\rm I_N \equiv\sup_{0<\lz<3N}\lz^p\rho(\{x\in S:\,
\cm_{0,1/2}^\cp f(x)>\lz\})$. Combining Lemma \ref{l3.1}(vi) with
 $f\in L^{p_0,\fz}$ and $p\ge p_0$ implies that $\rm I_N<\infty$.
Thus, by Lemma \ref{l3.5},
\begin{eqnarray*}
\rm I_N&&=3^p\sup_{0<\lz<N}\lz^p\rho\lf(\lf\{x\in S:\, \cm_{0,1/2}^\cp f(x)>3\lz\r\}\r)\\
&&\le 3^p\sup_{0<\lz<N}\lz^p\rho\lf(\lf\{x\in S:\,
\cm_{0,1/2}^\cp f(x)>3\lz,\, \cm_{0, s}^\sharp f(x)\le \lz/ 4\r\}\r)\\
&&\quad+3^p\sup_{0<\lz<N}\lz^p\rho\lf(\lf\{x\in S:\,
\cm_{0, s}^\sharp f(x)>\lz/4\r\}\r)\\
&&\le 2sC_23^p\sup_{0<\lz<N}\lz^p\rho\lf(\lf\{x\in S:\,
\cm_{0,1/2}^\cp f(x)>\lz\r\}\r)
+3^p 4^p\|\cm_{0, s}^\sharp f\|_{L^{p,\,\fz}}^p\\
&&\le s23^pC_2{\rm I}_N +3^p 4^p\|\cm_{0, s}^\sharp f\|_{L^{p,\,\fz}}^p.
\end{eqnarray*}
Since $\rm I_N$ is finite, the assumption $s< (2^23^p C_2)^{-1} $ implies that $\rm I_N\ls
\|\cm_{0, s}^\sharp f\|_{L^{p,\,\fz}}^p$. Letting $N\to\fz$ and
using Lemma \ref{l3.1}(vi) yield the desired conclusion.
\end{proof}

\begin{rem}\rm
A Fefferman--Stein inequality involving the sharp maximal function also holds. More precisely, let $p_0$ be in $(1,\infty)$. Then for any $p\in(p_0, \infty)$, by Lemma \ref{l3.4}(iv) and Proposition \ref{p3.1},
\begin{eqnarray}\label{3.15}
\|f\|_{L^p} \le C \|f^\sharp\|_{L^p} \quad\quad \forall\
f\in L^{p_0, \infty}\,.
\end{eqnarray}
As we shall see, such a Fefferman--Stein inequality is not enough for the
proof of Theorem \ref{t4.1} below. This explains why
we studied Fefferman--Stein type inequalities related to the local
sharp maximal function as in Proposition \ref{p3.2},
which is crucial in the proof of
Theorem \ref{t4.1}. However, \eqref{3.15} would be enough to
give a direct proof of Theorem \ref{t4.2} below.
\end{rem}

\section{Maximal Singular Integrals}\label{s4}

In this section, we consider the boundedness of
maximal singular integrals whose kernels satisfy some integral size
condition and H\"ormander's condition.

Assume that $K$ is a locally integrable function on $(S\times S)\setminus\{(x, x):\, x\in S\}$ such that
\begin{equation}\label{4.1}
\sup_{y\in S}\,\sup_{r>0}\dint_{r<d(x, y)\le2r}\big[|K(x, y)|+|K(y, x)|\big]\,d\rho(x) \equiv
\nu_1<\fz,
\end{equation}
and
\begin{equation}\label{4.2}
\sup_{R\in\crz}\sup_{y,\,y'\in R}
\dint_{(R^\ast)^\complement}\lf[|K(x, y)-K(x, y')| +|K(y, x)-K(y',
x)|\r]\,d\rho(x) \equiv  \nu_2<\fz,
\end{equation}
where, for any $R$ in $\mathcal R$, $R^\ast \equiv \{x\in S:\,d(x, R)<r_R\}$ and
$(R^\ast)^\complement \equiv  S\setminus R^\ast$.

Let $T$ be the \emph{linear operator associated with a kernel} $K$ satisfying
\eqref{4.1} and \eqref{4.2}; in particular, for all $f\in L_c^\fz$ and
$x\notin\supp f$,
\begin{equation*}
Tf(x)=\dint_SK(x, y)f(y)\,d\rho(y).
\end{equation*}
We define the {\em maximal singular integral operator} $T^\ast$ by
\begin{equation}\label{4.3}
T^\ast f(x) \equiv \sup_{R\in\crz(x)}|T_R f(x)| \quad\quad
\forall\ f\in L_c^\fz,\, \forall\ x\in S,
\end{equation}
where $T_R$ is the {\em truncated operator} defined by
\begin{equation*}
    T_R f(x) \equiv \dint_{(R^\ast)^\complement} K(x, y)f(y)\,d\rho(y)
    \quad\quad \forall\ x\in S,\, \forall\ R\in\crz(x).
\end{equation*}

The main result concerning such maximal singular integrals is the following.

\begin{thm}\label{t4.1}
Suppose that $T^\ast$ is the maximal singular integral operator as in
\eqref{4.3} associated with a kernel $K$ satisfying \eqref{4.1} and
\eqref{4.2}. The following statements are equivalent:
\begin{enumerate}
\vspace{-0.2cm}
\item[\rm(i)] $T^\ast$ is bounded from $L_c^\fz$ to $\BMO$;
\vspace{-0.25cm}
\item[\rm(ii)] $T^\ast$ is bounded on $L^p$ for all $p\in(1, \fz)$;
\vspace{-0.25cm}
\item[\rm(iii)] $T^\ast$ is bounded on $L^p$ for some $p\in(1, \fz)$;
\vspace{-0.25cm}
\item[\rm(iv)] $T^\ast$ is bounded from $L^1$ to $L^{1,\fz}$.
\end{enumerate}
\end{thm}

To prove Theorem \ref{t4.1}, we first establish the following lemma
by following some ideas used by Grafakos \cite[Lemma~1]{gra}.

\begin{lem}\label{l4.1}
Let $T^\ast$ and $K$ be as in Theorem \ref{t4.1}. Suppose that
$p\in[1, \fz)$, $\lz>0$, $b\equiv\sum_{i\in I} b_i$ and $\{R_i\}_{i\in I}\subset \crz$
are such that for a fixed positive constant $C_3$ and all $i\in I$, $\supp
b_i\subset R_i$, $\int_Sb_i\,d\rho=0$, $\|b_i\|_{L^p}\le
C_3\lz[\rho(R_i)]^{1/p}$, and $\{R_i\}_{i\in I}$ are
pairwise disjoint. Then
\begin{equation}\label{4.4}
\rho\lf(\lf\{x\notin \bigcup_{i\in I} R_i^\ast:\, T^\ast
b(x)>(C_{\kz_0}C_3\nu_1+3)\lz\r\}\r)\le C_3\lf[\nu_1\wz
C_{\kz_0}+3\nu_2\r] \sum_{i\in I}\rho(R_i),
\end{equation}
where $\kz_0$ is the constant which appears in Lemma \ref{l2.1}, $C_{\kz_0} \equiv  3+\log_2(\kz_0+1)$ and $\wz C_{\kz_0} \equiv
2+\log_2[(4\kz_0+3)\kz_0]$.
\end{lem}

\begin{proof}
For any fixed $x\notin \bigcup_{i\in I} R_i^\ast$ and $R\in\crz(x)$,
we set ${\rm I}_1(x, R) \equiv \lf\{i\in I:\, R_i\subset R^\ast\r\},$ $
{\rm I}_2(x, R) \equiv \lf\{i\in I:\, R_i\cap R^\ast=\emptyset\r\}$ and
${\rm I}_3(x, R) \equiv \{i\in I:\, R_i\cap R^\ast\neq\emptyset,\, R_i\cap
(R^\ast)^\complement\neq\emptyset\}$. Then, by the definition of
$T_R$,
\begin{equation}\label{4.5}
\lf|T_R b(x)\r|\le \lf|T_R\lf(\sum_{i\in {\rm I}_2(x,
R)}b_i\r)(x)\r| +\lf|T_R\lf(\sum_{i\in {\rm I}_3(x,
R)}b_i\r)(x)\r|.
\end{equation}
Denote by $x_i$ the center of $R_i$. By $\int_Sb_i\,d\rho=0$
and $\supp b_i\subset R_i$, we obtain
\begin{eqnarray*}
\lf|T_R\lf(\sum_{i\in {\rm I}_2(x, R)}b_i\r)(x)\r|
&&=\lf|\sum_{i\in {\rm I}_2(x, R)}\dint_{R_i}[K(x, y)-K(x, x_i)] b_i(y)\,d\rho(y)\r|\\
&&\le\sum_{i\in I}\dint_{R_i}|K(x, y)-K(x, x_i)||b_i(y)|\,d\rho(y)
 \equiv  {\rm Z}_1(x).\noz
\end{eqnarray*}
To estimate the second term in the right-hand side of \eqref{4.5},
for any $i\in {\rm I}_3(x, R)$, we set
$$c_i(R) \equiv  \frac1{\rho(R_i)}
\int_{R_i}b_i(y)\chi_{_{(R^\ast)^\complement}}(y)\,d\rho(y).$$
Notice that
\begin{eqnarray}\label{4.6}
\lf|T_R\lf(\sum_{i\in {\rm I}_3(x, R)}b_i\r)(x)\r|
&&=\lf|\sum_{i\in {\rm I}_3(x, R)}\int_{R_i} K(x, y)b_i(y)
\chi_{_{(R^\ast)^\complement}}(y)\,d\rho(y)\r|\noz\\
&&\le\lf|\sum_{i\in {\rm I}_3(x, R)}\int_{R_i} K(x, y)\lf[b_i(y)
\chi_{_{(R^\ast)^\complement}}(y)-c_i(R)\r]\,d\rho(y)\r|\noz\\
&&\quad+\lf|\sum_{i\in {\rm I}_3(x, R)}\int_{R_i} K(x,
y)c_i(R)\,d\rho(y)\r|\equiv  {\rm J}_1(x)+{\rm J}_2(x).
\end{eqnarray}
For all $i\in {\rm I}_3(x, R)$, by the hypothesis $\|b_i\|_{L^p}\le
C_3\lz[\rho(R_i)]^{1/p}$, we have $|c_i(R)|\le C_3\lz$. From this,
it follows that
\begin{eqnarray*}
{\rm J}_1(x)&&=\lf|\sum_{i\in {\rm I}_3(x, R)}\int_{R_i}\lf[ K(x,
y)-K(x, x_i)\r]\lf[b_i(y)
\chi_{_{(R^\ast)^\complement}}(y)-c_i(R)\r]\,d\rho(y)\r|\\
&&\le {\rm Z}_1(x)+ C_3\lz\sum_{i\in I}\int_{R_i}| K(x, y)-K(x,
x_i)|\,d\rho(y)\equiv {\rm Z}_1(x)+{\rm Z}_2(x).\noz
\end{eqnarray*}
To estimate ${\rm J}_2(x)$, set ${\rm I}_3^{(1)}(x, R) \equiv \{i\in
{\rm I}_3(x, R):\, r_R\ge4\kz_0r_{R_i}\}$ and
$${\rm I}_3^{(2)}(x, R) \equiv \{i\in {\rm I}_3(x, R):\, r_R<4\kz_0r_{R_i}\}.$$
For any $i\in {\rm I}_3(x, R)$, assume that $w_i\in R_i\cap R^\ast$ and
$u_i\in R_i\cap (R^\ast)^\complement$. If $i\in {\rm I}_3^{(1)}(x, R)$, by
the assumption that $x\in R$ and Lemma \ref{l2.1}(i), we obtain
that for all $y\in R_i$,
$$d(x, y)\le d(x, w_i)+d(w_i, y)<(2\kz_0+1)r_R+2\kz_0r_{R_i}<2(\kz_0+1)r_R$$
and
$$d(x, y)\ge d(x, u_i)-d(u_i, y)>r_R-2\kz_0r_{R_i}\ge r_R/2.$$
These facts together with the pairwise disjointness of
$\{R_i\}_{i\in I}$ yield
\begin{eqnarray*}
\lf|\sum_{i\in {\rm I}_3^{(1)}(x, R)}\int_{R_i} K(x, y)c_i(R)\,d\rho(y)\r|
&&\le C_3\lz\sum_{i\in {\rm I}_3^{(1)}(x, R)}\int_{R_i}| K(x, y)|\,d\rho(y)\\
&&\le C_3\lz\int_{r_R/2<d(x, y)\le2(\kz_0+1)r_R}| K(x, y)|\,d\rho(y)\\
&&< [3+\log_2(\kz_0+1)]C_3\nu_1\lz.
\end{eqnarray*}
If $i\in {\rm I}_3^{(2)}(x, R)$ and since $x\in R\setminus \cup_{i\in
I}R_i^\ast$, for all  $y\in R_i$,
$$r_{R_i}<d(x, y)\le d(x, w_i)+d(w_i, y)\le
(2\kz_0+1)r_R+2\kz_0r_{R_i}<2(4\kz_0+3)\kz_0r_{R_i},$$
which implies that
\begin{eqnarray*}
\lf|\sum_{i\in {\rm I}_3^{(2)}(x, R)}\int_{R_i} K(x, y)c_i(R)\,d\rho(y)\r|
&&\le C_3\lz\sum_{i\in I}\int_{r_{R_i}<d(x, y)\le
2(4\kz_0+3)\kz_0r_{R_i}}
| K(x, y)|\chi_{_{R_i}}(y)\,d\rho(y)\\
&& \equiv  {\rm Z}_3(x).
\end{eqnarray*}
Therefore, ${\rm J}_2(x)\le [3+\log_2(\kz_0+1)]C_3\nu_1\lz+{\rm
Z}_3(x)$.

Set $C_{\kz_0} \equiv  3+\log_2(\kz_0+1)$. The estimates of ${\rm
J}_1(x)$ and ${\rm J}_2(x)$ together with \eqref{4.5} and
\eqref{4.6} imply that for all $x\notin\cup_{i\in
I}R_i^\ast$,
\begin{equation}\label{4.7}
T^\ast b(x)\le 2{\rm Z}_1(x)+{\rm Z}_2(x)+C_{\kz_0}C_3\nu_1\lz+{\rm
Z}_3(x).
\end{equation}
By this, \eqref{4.1} and \eqref{4.2}, we obtain
\begin{eqnarray}\label{4.8}
&&\rho\lf(\lf\{x\notin \bigcup_{i\in I} R_i^\ast:\,
T^\ast b(x)>(C_{\kz_0}C_3\nu_1+3)\lz\r\}\r)\noz\\
&&\hs\le\rho\lf(\lf\{x\notin \bigcup_{i\in I} R_i^\ast:\,{\rm
Z}_1(x)>\lz/2\r\}\r)
+\rho\lf(\lf\{x\notin \bigcup_{i\in I} R_i^\ast:\,{\rm Z}_2(x)>\lz\r\}\r)\noz\\
&&\hs\quad+\rho\lf(\lf\{x\notin \bigcup_{i\in I}
R_i^\ast:\,{\rm Z}_3(x)>\lz\r\}\r)\noz\\
&&\hs\le\dfrac{2}{\lz}\sum_{i\in I}\dint_{S\setminus
R_i^\ast}\dint_{R_i}
|K(x, y)-K(x, x_i)||b_i(y)|\,d\rho(y)\,d\rho(x)\noz\\
&&\hs\quad+C_3\sum_{i\in I}\dint_{S\setminus R_i^\ast}
\int_{R_i}| K(x, y)-K(x, x_i)|\,d\rho(y)\,d\rho(x)\noz\\
&&\hs\quad+C_3\sum_{i\in I}\dint_{S\setminus R_i^\ast}
\int_{r_{R_i}<d(x, y)\le 2(4\kz_0+3)\kz_0r_{R_i}}
| K(x, y)|\chi_{_{R_i}}(y)\,d\rho(y)\,d\rho(x) \noz\\
&&\hs\le\dfrac{2\nu_2}{\lz}\sum_{i\in I}\|b_i\|_{L^1}+
{C_3}{\nu_2}\sum_{i\in I}\rho(R_i)
+{C_3}{\nu_1}\Big(2+\log_2[(4\kz_0+3)\kz_0]\Big)\sum_{i\in
I}\rho(R_i).
\end{eqnarray}
The hypothesis  $\|b_i\|_{L^p}\le
C_3\lz[\rho(R_i)]^{1/p}$ implies that
$$\sum_{i\in I}\|b_i\|_{L^1}\le \sum_{i\in I}\|b_i\|_{L^p}
[\rho(R_i)]^{1-1/p} \le C_3\lz\sum_{i\in I}\rho(R_i).$$ Combining
all these facts yields \eqref{4.4}. This finishes the proof.
\end{proof}

\begin{lem}\label{l4.2}
Let $T^\ast$ and $K$ be as in Theorem \ref{t4.1}. If $T^\ast$ is
bounded from $L_c^\fz$ to $\BMO$, then $T^\ast f\in L^{1,\fz}$
for all $f\in L_{c, 0}^\fz$.
\end{lem}

\begin{proof}
Without loss of generality, we may assume that $\supp f\subset
R_0\in\crz$. Since $f\in  L_{c, 0}^\fz$, we have that $T^\ast
f\in\BMO$, which implies the local integrability of $T^\ast f$ and
$$\sup_{\az>0}\az \rho\lf(\lf\{x\in (R_0)^\ast:\,
T^\ast f(x)>\az\r\}\r)\le \|(T^\ast
f)\chi_{(R_0)^\ast}\|_{L^1}<\fz.$$ Therefore, it suffices to prove
that
\begin{eqnarray}\label{4.9}
&&\sup_{\az>0}\az \rho\lf(\lf\{x\notin (R_0)^\ast:\, T^\ast
f(x)>\az\r\}\r)<\fz.
\end{eqnarray}
To this end, fix $x\notin (R_0)^\ast$ and $R\in\crz(x)$. If
$R_0\subset R^\ast$, then $T_R f(x)=0$. If $R_0\cap
R^\ast=\emptyset$, then we denote by $x_0$ the center of $R_0$ and
use the fact that $\int_S f\,d\rho=0$ to obtain
\begin{eqnarray*}
|T_R f(x)|&&\le \int_{R_0} |K(x, y)-K(x, x_0)|
|f(y)|\,d\rho(y) \equiv {\rm Y}_1(x).
\end{eqnarray*}
If $R_0\cap R^\ast\neq\emptyset$ and $R_0\cap
(R^\ast)^\complement\neq\emptyset$, we take a Calder\'on--Zygmund decomposition
of $f$ at level $\az$ and
write $f=g^\az+b^\az$ with $b^\az\equiv\sum_{i\in I_\az} b_i^\az$,
$\|g^\az\|_{L^\fz}\le C_1\az$, $\supp b_i^\az\subset R_i^\az$,
$\{R_i^\az\}_{i\in I_\az}\subset\crz$ are mutually disjoint,
$\int_Sb_i^\az\,d\rho=0$, $\az\le
\frac1{\rho(R_i^\az)}\int_{R_i^\az}|f|\,d\rho\le {C_1}\az$ and
$\|b_i^\az\|_{L^1}\le C_1\az\rho(R_i^\az)$, where $I_\az$ is
a certain index set and $C_1$ is the constant which appears in Proposition \ref{p2.2}.
Then
\begin{eqnarray*}
|T_R f(x)|&&=\lf| \int_{R_0\cap (R^\ast)^\complement} K(x, y)f(y)\,d\rho(y)\r|\\
&&\le \lf| \int_{R_0\cap (R^\ast)^\complement} K(x,
y)g^\az(y)\,d\rho(y)\r|
+\sum_{i\in I_\az}\lf| \int_{R_0\cap (R^\ast)^\complement}
K(x, y)b_i^\az(y)\,d\rho(y)\r|\\
&& \equiv  {\rm Y}_2(x)+{\rm Y}_3(x).
\end{eqnarray*}
Observe that if $R_0\cap R^\ast\neq\emptyset$ and $R_0\cap
(R^\ast)^\complement\neq\emptyset$, then for all $y\in R_0\cap
(R^\ast)^\complement$,
$$ \frac{r_R+r_{R_0}}2\le d(x, y)\le d(x, y_1)+d(y_1, y)
\le (1+2\kz_0) r_R+2\kz_0 r_{R_0}<(2\kz_0+1)(r_R+r_{R_0}),$$
where $y_1$ is some fixed point in $R_0\cap R^\ast$. Thus, for
all $x\notin (R_0)^\ast$,
\begin{eqnarray*}
{\rm Y}_2(x)&&\le \int_{\frac{r_R+r_{R_0}}2\le d(x, y)<
(2\kz_0+1)(r_R+r_{R_0})} |K(x, y)||g^\az(y)|\,d\rho(y)
\le \nu_1 {C_1}\az\lf[3+\log_2(2\kz_0+1)\r].
\end{eqnarray*}
By proceeding as in the proof of Lemma \ref{l4.1}, we obtain that
for all $x\notin \cup_{i\in I_\az}(R_i^\az)^\ast$,
\begin{eqnarray*}
{\rm Y}_3(x)&&\le 2{\rm Z}_1'(x)+{\rm Z}_2'(x)
+C_{\kz_0} {C_1} \nu_1\az+{\rm Z}_3'(x),
\end{eqnarray*}
where $C_{\kz_0}$ is as in Lemma \ref{l4.1}, and
$$ {\rm Z}_1'(x) \equiv  \sum_{i\in I_\az}\dint_{R_i^\az}
|K(x, y)-K(x, x_i)||b_i^\az(y)|\,d\rho(y),$$
$${\rm Z}_2'(x) \equiv  {C_1}\az\sum_{i\in I_\az}
\int_{R_i^\az}|K(x, y)-K(x, x_i)|\,d\rho(y),$$
$${\rm Z}_3'(x) \equiv  {C_1}\az\sum_{i\in I_\az}
\int_{r_{R_i^\az}<d(x, y)\le 2(4\kz_0+3)\kz_0r_{R_i^\az}}
| K(x, y)|\chi_{_{R_i^\az}}(y)\,d\rho(y).$$
Combining all these estimates, we obtain that for all $x\notin \cup_{i\in I_\az}(R_i^\az)^\ast$
$$T^\ast f(x) \le{\rm Y}_1(x)+\nu_1 {C_1} \az[3+\log_2(2\kz_0+1)]+ 2{\rm Z}_1'(x)+{\rm Z}_2'(x) +C_{\kz_0} {C_1}\nu_1\az+ {\rm Z}_3'(x),$$
and hence for all $\az>0$,
\begin{eqnarray*}
&&\rho\big( \{x\notin (R_0)^\ast:\, T^\ast f(x)>(C_{\kz_0}\nu_1 {C_1}+4+ \nu_1 {C_1}[3+\log_2(2\kz_0+1)])\az \}\big)\\
&&\hs\le \rho\big( \cup_{i\in I_\az}(R_i^\az)^\ast\big )+ \rho\big( \{x\notin (R_0)^\ast:\,{\rm Y}_1(x)>\az \}\big) \\
&&\hs \quad+ \rho\big( \{x\notin \cup_{i\in I_\az}(R_i^\az)^\ast:\,
2{\rm Z}_1'(x)+{\rm Z}_2'(x) +{\rm Z}_3'(x)>3\az \}\big).
\end{eqnarray*}
By Lemma \ref{l2.1}(ii), the condition $\az\le
\frac1{\rho(R_i^\az)}\int_{R_i^\az}|f|\,d\rho$ and
the pairwise disjointness of $\{R_i^\az\}_{i\in I_\az}$, we have
$ \rho\lf( \cup_{i\in I_\az}(R_i^\az)^\ast\r)
\le\kz_0 \sum_{i\in I_\az} \rho(R_i^\az)\le\kz_0
\az^{-1}\|f\|_{L^1},$ where $\kz_0$ is the constant which appears in Lemma \ref{l2.1}.
For the second term, applying \eqref{4.2} yields
\begin{eqnarray*}
\rho\lf(\lf\{x\notin (R_0)^\ast:\,{\rm Y}_1(x)>\az\r\}\r)
&&\le \az^{-1}\int_{x\notin (R_0)^\ast} \int_{R_0} |K(x, y)-K(x, x_0)|
|f(y)|\,d\rho(y)\,d\rho(x)\\
&&\le \az^{-1}\nu_2\|f\|_{L^1}.
\end{eqnarray*}
Finally, an argument similar to \eqref{4.8}
yields that
$$ \rho\lf(\lf\{x\notin \cup_{i\in I_\az}(R_i^\az)^\ast:\,
2{\rm Z}_1'(x)+{\rm Z}_2'(x) +{\rm Z}_3'(x)>3\az\r\}\r) \le {C_1}\lf[\nu_1\wz
C_{\kz_0}+3\nu_2\r] \|f\|_{L^1},$$
where $\wz C_{\kz_0}$ is the constant which appears in Lemma \ref{4.1}.
Combining all these estimates gives \eqref{4.9}. Hence, $T^\ast f\in L^{1,\fz}$.
\end{proof}

\begin{proof}[Proof of Theorem \ref{t4.1}]
To show that (i) implies (ii), by the Marcinkiewicz interpolation theorem
(see \cite[Theorem~1.4.19]{g}) and the fact that
$L_{c,0}^\fz$ is dense in $L^q$ when $q\in(1,\infty)$ (see \cite[Lemma~5.3]{va}),
it suffices to show that
for any $p\in(1, \fz)$ and all $f\in L_{c,0}^\fz$,
\begin{equation}\label{4.10}
   \|T^\ast f\|_{L^{p,\infty}}\ls\|f\|_{L^p}.
\end{equation}

To prove \eqref{4.10}, we fix $p\in(1, \infty)$ and $f\in L_{c,0}^\fz$.
By applying Proposition \ref{p2.2} and
Remark \ref{r2.1}, we know that for any given $\az>0$, there exist a
positive constant $C_1$ and a sequence of pairwise disjoint sets
$\{R_i^\az\}_{i\in I_\az}\subset\crz$ such that $f$ is decomposed into $f=g^\az+b^\az=g^\az+\sum_{i\in
I_\az} b_i^\az$ such that
\begin{enumerate}
\vspace{-0.25cm}
\item[(a)] $|g^\az(x)|\le {C_1}\az$ for almost all $x\in S$ and
$\|g^\az\|_{L^\fz}\le\|f\|_{L^\fz}$;
\vspace{-0.25cm}
\item[(b)] for all $i\in I_\az$, $\supp b_i^\az\subset R_i^\az\in\crz$
and $\int_S b_i^\az\,d\rho=0$;
\vspace{-0.25cm}
\item[(c)] for all $i\in I_\az$,
$\az\le\{\frac1{\rho(R_i^\az)}\int_{R_i^\az}|f|^p\,d\rho\}^{1/p}\le {C_1}\az$;
\vspace{-0.25cm}
\item[(d)] for all $i\in I_\az$, $\|b_i^\az\|_{L^p}
\le {C_1}\az[\rho(R_i^\az)]^{1/p}$,
\vspace{-0.25cm}
\end{enumerate}
where $I_\az$ is an index set and $C_1$ is the constant which appears in Proposition \ref{p2.2}.
By proceeding as in the proof of
\cite[(3.5)]{hyy}, we obtain that for all $s\in(0, 1)$, $\az>0$, and $x\in S$,
\begin{equation*}
\cm_{0, s}^\sharp(T^\ast f)(x) =\cm_{0, s}^\sharp\big(T^\ast(g^\az+b^\az)\big)(x) \le \cm_{0, s/2}^\sharp
(T^\ast g^\az)(x)+\cm_{0, s/2}(T^\ast b^\az)(x);
\end{equation*}
therefore, with $C_{\kz_0}$ as in Lemma \ref{l4.1} and $C_4 \equiv  \|T^\ast\|_{L_c^\infty\to\BMO}$,
\begin{eqnarray*}
&&\rho(\{x\in S:\, \cm_{0, s}^\sharp (T^\ast f)(x)>(2C_4{C_1} s^{-1}+C_{\kz_0}{C_1}\nu_1+3)\az\})\\
&&\hs\le\rho(\{x\in S:\, \cm_{0, s/2}^\sharp(T^\ast
g^{\az})(x)>2C_4{C_1} s^{-1}\az\})\nonumber\\
&&\quad\hs+\rho(\{x\in S:\, \cm_{0, s/2}(T^\ast b^{\az})(x)>(C_{\kz_0}{C_1}\nu_1+3)\az\})
 \equiv {\rm I}+{\rm II}.\nonumber
\end{eqnarray*}
By (i), Lemma \ref{l3.4}(iv) and Property (a), we obtain
\begin{equation*}
  \|\cm_{0, s/2}^\sharp(T^\ast g^{\az})\|_{L^\fz}\le 2s^{-1} \|(T^\ast g^{\az})^\sharp\|_{L^\fz}
  \le 2s^{-1}\|T^\ast\|_{L_c^\infty\to\BMO} \|g^{\az}\|_{L^\fz}\le 2C_4{C_1} s^{-1}\az,
\end{equation*}
which implies that ${\rm I}=0$.
By  Lemma \ref{l3.1}(vii),  the term ${\rm II}$ can be estimated by
\begin{eqnarray*}
{\rm II}&&\le 2\|\cm\|_{L^1\to L^{1,\fz}} s^{-1}\rho\lf(\lf\{x\in
S:\, T^\ast b^{\az}(x)>(C_{\kz_0}{C_1}\nu_1+3)\az\r\}\r),
\end{eqnarray*}
then, applying Lemma \ref{l4.1}, Lemma \ref{l2.1}(ii)  and Property (c) yields that
\begin{eqnarray*}
{\rm II}&&\ls s^{-1}\lf[\rho\Big(\bigcup_{i\in
I_{\az}}\lf(R_i^{\az}\r)^\ast\Big)+\rho\Big(\Big\{x\notin
\bigcup_{i\in I_{\az}}\lf(R_i^{\az}\r)^\ast:\,
 T^\ast b^{\az}(x)>(C_{\kz_0}C_3\nu_1+3)\az\Big\}\Big)\r]\noz\\
&&\ls s^{-1} \sum_{i\in I_{\az}}\rho\lf(R_i^{\az}\r)\ls s^{-1} \az^{-p}\|f\|_{L^p}^p.\noz
\end{eqnarray*}
The estimates of ${\rm I}$ and ${\rm II}$ above imply that for any given $s\in(0, 1)$,
\begin{equation}\label{4.11}
\|\cm_{0, s}^\sharp (T^\ast f)\|_{L^{p,\fz}}\ls s^{-1/p}\|f\|_{L^p}.
\end{equation}
By the assumption $f\in L_{c,0}^\infty$ and Lemma \ref{l4.2}, $T^*f\in L^{1,\fz}$. Then by applying
 Proposition \ref{p3.2} to the function $T^*f$ and by \eqref{4.11}, we obtain that for $s\in(0, 1/2]$
such that $s<(2^23^p C_2)^{-1} $,
\begin{eqnarray*}
\|T^\ast f\|_{L^{p,\fz}} &&\le C \|\cm_{0, s}^\sharp (T^\ast
f)\|_{L^{p,\fz}}\le C \|f\|_{L^p}.
\end{eqnarray*}
This proves \eqref{4.10}. Thus, (ii) holds.

It is obvious that (ii) implies (iii). Now we assume that (iii) holds for
an index $p\in (1,\fz)$ and show that
(iv) holds. To this end, for any given $f\in L^1$ and $\az>0$,
we use Proposition \ref{p2.2}
to obtain a sequence of mutually disjoint sets,
$\{R_i^\az\}_{i\in I_\az}\subset\crz$, and a decomposition of $f$ as
$f=g^\az+b^\az=g^\az+\sum_{i\in I_\az}b_i^\az$ where $\|g^\az\|_{L^\fz}\ls\az$, every $b_i^\az$ is supported on $R_i^\az$
and has integral $0$,
$\frac1{\rho(R_i^\az)}\int_{R_i^\az}|f|\,d\rho\approx\az$ and
$\|b_i^\az\|_{L^1}\ls\az\rho(R_i^\az)$.
By Lemma \ref{l4.1} and Lemma \ref{l2.1}(ii),
there exists a sufficiently large positive constant $C$ such that for all $\az>0$,
\begin{eqnarray*}
&&\rho\big(\lf\{x\in S:\, T^\ast f(x)>(C+1) \az\r\}\big)\\
&&\hs\le
\rho\big(\lf\{x\in S:\, T^\ast g^\az(x)>\az\r\}\big)
+\rho\lf(\bigcup_{i\in I_\az}\lf(R_i^\az\r)^\ast\r)\\
&&\hs\quad+\rho\lf(\lf\{x\notin \bigcup_{i\in
I_\az}\lf(R_i^\az\r)^\ast:\,
T^\ast b^\az(x)>C\az\r\}\r)\\
&&\hs \ls \az^{-p}\|T^\ast g^\az\|_{L^p}^p+\sum_{i\in
I_\az}\rho(R_i^\az).
\end{eqnarray*}
Notice that $\sum_{i\in I_\az}\rho(R_i^\az)\ls \az^{-1}\|f\|_{L^1}$.
Using the $L^p$-boundedness of $T^\ast$ (by (iii)) and the
properties of $g^\az$, we  have
\begin{eqnarray*}
\az^{-p}\|T^\ast g^\az\|_{L^p}^p \ls \az^{-p} \|T^\ast\|_{L^p\to
L^p}^p\|g^\az\|_{L^p}^p
\ls\az^{-1}\|g^\az\|_{L^1}\ls\az^{-1}\|f\|_{L^1}.
\end{eqnarray*}
Combining all the above estimates yields that
$\|T^\ast f\|_{L^{1,\fz}}\ls\|f\|_{L^1}$.
Hence, (iv) holds.

Finally, we show that (iv) implies (i). Fix $\sz\in(0, 1)$.
By Lemma \ref{l3.3}, it suffices to prove that for all  $f\in L_c^\fz$,
$\|T^\ast f\|_{\ast,\,\sz}\ls\|f\|_{L^\fz}$. To this end, for any given
$f\in L_c^\fz$ and $R\in\crz$, we decompose $f$ into
$f=f\chi_{_{R^\ast}}+f\chi_{_{S\setminus R^\ast}} \equiv  f_1+f_2$.
Notice that for all $c\in \cc$ and $x\in S$, $\lf|T^\ast
f(x)-c\r|\le T^\ast f_1(x)+\lf|T^\ast f_2(x)-c\r|.$ Then, for all
$R\in\crz(x)$,
\begin{eqnarray}\label{4.12}
&&\inf_{c\in\cc}\dfrac1{\rho(R)}\dint_R|T^\ast f(x)-c|^\sz\,d\rho(x)\noz\\
&&\hs\le\dfrac1{\rho(R)}\dint_R|T^\ast f_1(x)|^\sz\,d\rho(x)
+\inf_{c\in\cc}\dfrac1{\rho(R)}\dint_R|T^\ast
f_2(x)-c|^\sz\,d\rho(x)  \equiv  {\rm Z}_1+{\rm Z}_2.
\end{eqnarray}
Using $\sz\in(0, 1)$ and the hypothesis that $T^\ast$ is bounded
from $L^1$ to $L^{1,\fz}$ together with Lemma \ref{l2.1}(ii), we
obtain
\begin{eqnarray*}
{\rm Z}_1 &&=\dfrac1{\rho(R)}\dint_0^\fz \sz t^{\sz-1}\rho\lf(\lf\{
x\in R:\, T^\ast f_1(x)>t\r\}\r)\,dt\\
&&\le \dfrac1{\rho(R)}\lf\{\dint_0^{\|f\|_{L^\fz}} \sz
t^{\sz-1}\rho(R)\,dt +\dint_{\|f\|_{L^\fz}}^\fz \sz t^{\sz-1}
\dfrac{\|T^\ast\|_{L^1\to L^{1,\fz}}\|f_1\|_{L^1}}{t}\,dt\r\}\ls \|f\|_{L^\fz}^{\sz}.
\end{eqnarray*}
By this and \eqref{4.12}, the proof of (i) is reduced to the estimate ${\rm Z}_2\ls \|f\|_{L^\fz}^\sz$.
Since $f_2\in L_c^\infty$, an argument similar to the
one used for ${\rm Z}_1$ yields $|T^\ast f_2|^\sz\in L_\loc^1$;
thus there exists some $z_R\in R$ such that $T^\ast
f_2(z_R)<\fz$. Notice that for all $x\in R$ and $\wz R\in\crz(x)$,
Lemma \ref{l2.1}(i) implies that $\{y\in S:\, d(y, x)>r_{\wz R}\}
=({\wz R}^\ast)^\complement \cup \{ y\in {\wz R}^\ast: \,  d(y, x)>r_{\wz R}\}$.
From this, it follows that for all $x\in R$, we can write $T^\ast f_2(x)$ as follows:
\begin{eqnarray}\label{4.13}
T^\ast f_2(x)=&&\sup_{\wz R\in\crz(x)}\lf|\dint_{d(y, x)>r_{\wz
R}}K(x, y)f_2(y)\,d\rho(y)\r.\noz\\
&&\lf. -\dint_{\gfz{d(y, x)>r_{\wz R}}{y\in \wz
R^\ast}}K(x, y)f_2(y)\,d\rho(y)\r|.
\end{eqnarray}
In particular, the equality \eqref{4.13} holds for $T^\ast f_2(z_R)$. Thus, for all $x\in R$,
\begin{eqnarray*}
&&|T^\ast f_2(x)-T^\ast f_2(z_R)|\\
&&\hs\le\lf|\sup_{\wz R\in\crz(x)}\bigg|\dint_{d(y, x)>r_{\wz R}}
K(x, y)f_2(y)\,d\rho(y)\bigg| -\sup_{\wz
R\in\crz(z_R)}\bigg|\dint_{d(y, z_R)>r_{\wz R}}
K(z_R, y)f_2(y)\,d\rho(y)\bigg|\r|\\
&&\hs\quad+\sup_{\wz R\in\crz(x)} \lf|\dint_{\gfz{d(y, x)>r_{\wz
R}}{y\in \wz R^\ast}}K(x, y)f_2(y)\,d\rho(y)\r| +\sup_{\wz
R\in\crz(z_R)}
\lf|\dint_{\gfz{d(y, z_R)>r_{\wz R}}{y\in \wz R^\ast}}K(z_R, y)f_2(y)\,d\rho(y)\r|\\
&&\hs \equiv  {\rm L}_{1}+{\rm L}_{2}+{\rm L}_{3}.
\end{eqnarray*}
To see this, by symmetry, it suffices to show that
$T^\ast f_2(x)-T^\ast f_2(z_R)\le L_1+L_2+L_3$,
which follows by first writing $T^\ast f_2(x)$ and $T^\ast f_2(z_R)$
as in \eqref{4.13}, then applying
$$\sup_{i\in \Lambda}|a_i-b_i| \le \sup_{i\in \Lambda}|a_i|
+\sup_{i\in \Lambda}|b_i|$$
to the expression of  $T^\ast f_2(x)$, and
$\sup_{i\in \Lambda}|a_i-b_i|\ge \sup_{i\in \Lambda}|a_i|-\sup_{i\in \Lambda}|b_i|$
in the expression of $T^\ast f_2(z_R)$,
where $\Lambda$ denotes an index set which might be uncountable.

When $d(y, x)>r_{\wz R}$ and $y\in \wz R^\ast$,
by $x\in R\cap \wz R$ and Lemma \ref{l2.1}(i), we have
$r_{\wz R}<d(y, x)\le (2\kz_0+1)r_{\wz R}$,
which combined with \eqref{4.1} implies that
 ${\rm L}_{2}\ls \nu_1\|f\|_{L^\fz}$.
Similarly, ${\rm L}_{3}\ls\nu_1\|f\|_{L^\fz}$.
To estimate ${\rm L}_1$, by the properties of Calder\'on--Zygmund
sets, we obtain
$$\sup_{\wz R\in\crz(x)}\lf|\dint_{d(y, x)>r_{\wz R}}K(x, y)f_2(y)\,d\rho(y)\r|
=\sup_{\ez>0}\lf|\dint_{d(y, x)>\ez}K(x,
y)f_2(y)\,d\rho(y)\r|.$$
By this and the inequality $\sup_{i\in \Lambda}|a_i-b_i|\ge \big|\sup_{i\in \Lambda}
|a_i|-\sup_{i\in \Lambda}|b_i|\big|$, we obtain that
\begin{eqnarray*}
{\rm L}_{1}&&\le\sup_{\ez>0} \lf|\dint_{d(y, x)>\ez} K(x,
y)f_2(y)\,d\rho(y)
-\dint_{d(y, z_R)>\ez} K(z_R, y)f_2(y)\,d\rho(y)\r|\\
&&\le \sup_{\ez>0}\dint_{{d(x, y)>\ez}, \,{d(z_R, y)>\ez}}
| K(x, y)-K(z_R, y)||f_2(y)|\,d\rho(y)\\
&&\hs+\sup_{\ez>0}\dint_{d(x, y)>\ez\ge d(z_R, y)}|K(x, y)f_2(y)|\,d\rho(y)\\
&&\hs+\sup_{\ez>0}\dint_{ d(z_R, y)>\ez\ge d(x, y)}|K(z_R,
y)f_2(y)|\,d\rho(y)  \equiv  {\rm J}_1+{\rm J}_2+{\rm J}_3.
\end{eqnarray*}
From \eqref{4.2} and $\supp f_2\subset (R^\ast)^\complement$,
it follows that ${\rm J}_1\ls\nu_2\|f\|_{L^\fz}$.
If $x\in R$,   $y\notin R^\ast$ and $d(x, y)>\ez\ge d(z_R, y)$,
by Lemma \ref{l2.1}(i), we have $r_R<d(y, R)\le d(y,z_R)\le\ez $ and
$$ d(x, y)\le d(x, z_R)+d(z_R, y)<2\kz_0r_R+\ez<(2\kz_0+1)\ez,$$
which together with \eqref{4.1} implies that ${\rm
J}_2\ls\nu_1\|f\|_{L^\fz}$. Similarly, ${\rm
J}_3\ls\nu_1\|f\|_{L^\fz}$. Thus, ${\rm L}_1=\sum_{i=1}^3 {\rm J}_i\ls \|f\|_{L^\fz}$.
Combining the estimates of ${\rm L}_1, {\rm L}_2$ and ${\rm L}_3$ yields
 ${\rm Z}_2\ls\|f\|_{L^\fz}^\sz$. This finishes the proof of (iv) implies (i),
 and hence the proof of Theorem \ref{t4.1}.
\end{proof}

Applying Theorem \ref{t4.1} and the Calder\'on--Zygmund
decomposition, we obtain the following result.

\begin{thm}\label{t4.2}
Let $T$ be the integral operator associated with a kernel $K$ satisfying
\eqref{4.1} and \eqref{4.2}. If $T$ is bounded
 on $L^2$, then the maximal singular integral
$T^\ast$ defined as in \eqref{4.3} is bounded from $L^1$ to
$L^{1,\fz}$, from $L^p$ to $L^p$ for all $p\in(1, \fz)$, and
from $L_c^\fz$ to $\BMO$.
\end{thm}

To prove Theorem \ref{t4.2}, we need the following Cotlar-type
inequality.

\begin{lem}\label{l4.3}
Under the assumptions of Theorem \ref{t4.2}, for all $g\in L_c^\fz$ and $x\in S$,
\begin{equation}\label{4.14}
T^\ast g(x)\le \cm(Tg)(x) +\lf[\nu_2+\kz_0^{1/2}\|T\|_{L^2\to
L^2}\r]\|g\|_{L^\fz},
\end{equation}
where $\cm$ is the Hardy--Littlewood maximal operator defined in \eqref{2.1}.
\end{lem}

\begin{proof}
We only give an outline of the proof because of its similarity to the
argument used in \cite[Theorem~1]{gra}; see also \cite[p.\,295]{g}.
Indeed, for all $x\in S$,  $R\in\crz(x)$, $z\in R$  and all $g\in
L_c^\fz$, we use H\"ormander's condition \eqref{4.2} to obtain
\begin{eqnarray*}
|T_Rg(x)|
&&\le|T(g\chi_{(R^\ast)^\complement})(x)-T(g\chi_{(R^\ast)^\complement})(z)|
+|Tg(z)|+|T(g\chi_{R^\ast})(z)|\\
&&\le\nu_2\|g\|_{L^\fz}+|Tg(z)|+|T(g\chi_{R^\ast})(z)|.
\end{eqnarray*}
Taking the integral average over $R$ with respect to the variable
$z$ in both sides of this inequality yields that
$$|T_Rg(x)|\le\nu_2\|g\|_{L^\fz}+\dfrac1{\rho(R)}\dint_R|Tg(z)|\,d\rho(z)+
\frac1{\rho(R)}\dint_R|T(g\chi_{R^\ast})(z)|\,d\rho(z).$$
By \eqref{2.1},  H\"older's inequality, the
$L^2$-boundedness of $T$ and Lemma \ref{l2.1}(ii), we obtain
\begin{eqnarray*}
|T_Rg(x)|&&\le\nu_2\|g\|_{L^\fz}+\cm(Tg)(x)+
\lf\{\frac1{\rho(R)}\dint_R|T(g\chi_{R^\ast})(z)|^2\,d\rho(z)\r\}^{1/2}\\
&&\le  \cm(Tg)(x)+\lf[\nu_2+\kz_0^{1/2} \|T\|_{L^2\to
L^2}\r]\|g\|_{L^\fz},
\end{eqnarray*}
which completes the proof.
\end{proof}

\begin{proof}[Proof of Theorem \ref{t4.2}]
By Theorem \ref{t4.1}, it suffices to show that $T^\ast$ is bounded
from $L^1$ to $L^{1,\fz}$. To this end, for any $f\in L^1$
with bounded support and $\az>0$,
we decompose $f$ at level $\az$ into
$f=g^\az+b^\az=g^\az+\sum_{i\in I_\az} b_i^\az,$
where $I_\az$ is a certain index set,
$\|g^\az\|_{L^\fz}\le {C_1}\az$, $\supp b_i^\az\subset R_i^\az$,
$\{R_i^\az\}_{i\in I_\az}\subset\crz$ are mutually disjoint,
$\int_Sb_i^\az\,d\rho=0$, $\az\le
\frac1{\rho(R_i^\az)}\int_{R_i^\az}|f|\,d\rho\le {C_1}\az$ and
$\|b_i^\az\|_{L^1}\le {C_1}\az\rho(R_i^\az)$. Here $C_1$ is
the constant which appears in Proposition \ref{p2.2}. For
$C_0>C_{\kz_0}{C_1}\nu_1+3$ sufficiently large, which will be
determined later, we have
\begin{eqnarray*}
\rho\lf(\lf\{x\in S:\, T^\ast f(x)>C_0\az\r\}\r)
&&\le \rho\lf(\lf\{x\in S:\, T^\ast g^\az(x)>(C_0-C_{\kz_0}{C_1}\nu_1-3)\az\r\}\r)\\
&&\hs+\rho\lf(\lf\{x\in S:\, T^\ast
b^\az(x)>(C_{\kz_0}{C_1}\nu_1+3)\az\r\}\r) \equiv  {\rm Z}_1+{\rm Z}_2.
\end{eqnarray*}
Here $C_{\kz_0}$ is as in Lemma \ref{l4.1}. To estimate ${\rm Z}_1$, the inequality \eqref{4.14} applied to to $g^\az$ gives that
$$T^\ast g^\az(x)\le\cm(Tg^\az)(x)+{C_1}\lf[\nu_2+\kz_0^{1/2}
\|T\|_{L^2\to L^2}\r]\az.$$ Set $\wz C \equiv
{C_1}[\nu_2+\kz_0^{1/2}\|T\|_{L^2 \to
L^2}]+C_{\kz_0}{C_1}\nu_1+3$. If  $C_0>\wz C$, by the
$L^2$-boundedness of $\cm$ and $T$, the facts that
$\|g^\az\|_{L^\fz}\ls \az$ and $\|g^\az\|_{L^1}\ls\|f\|_{L^1}$, we obtain
$${\rm Z}_1\ls \az^{-2} \|\cm(Tg^\az)\|_{L^2}^2\ls
\az^{-2} \|g^\az\|_{L^2}^2\ls \az^{-1}\|f\|_{L^1}.$$
Lemma \ref{l4.1} implies that ${\rm Z}_2\ls\sum_{i\in
I_\az}\rho(R_i^\az)\ls \az^{-1}\|f\|_{L^1}$.
By the estimates of ${\rm Z}_1$ and ${\rm Z}_2$, we have that $T^\ast$
maps all $L^1$ functions with bounded support into $L^{1,\fz}$.
A standard density argument implies the boundedness of $T^\ast$
from $L^1$ to $L^{1,\fz}$. This concludes the proof.
\end{proof}

\section{Applications to Multiplier Operators on {$ax+b$\,--Groups}}\label{s5}

The aim of  this section is to apply the results in Section
\ref{s4} to the multipliers of a distinguished Laplacian $\Delta$ on
$(S, d, \rho)$. Let us begin with some known facts related to the
integration formulae and
spherical analysis on $S$; for details we refer the reader to
\cite{ADY, A95, dr, dr1, helgason}.

A {\em radial function} on $S$ is a function that depends only
on the distance from the identity.
If $f$ is radial and $f\in C_c^\infty(S)$, then we have
the following integration formula:
\begin{equation}\label{lambda-formula}
\int_S f(x)\,d\lz(x)=  C \int_0^\infty f(r) A(r)\,dr,
\end{equation}
where $C$ is a positive constant depending only on $S$,
$A(r)= 4^n \sinh ^n \lf(\frac r2\r) \cosh^n\lf(\frac r2\r)$ for all $ r>0$ and
$\lambda$ denotes the {\em left Haar measure}. One easily checks that
\begin{equation}\label{misA}
A(r)\ls \lf(\frac r{1+r}\r)^n e^{nr}\quad\quad\forall\ r>0.
\end{equation}
A radial function $\phi$ is {\em spherical } if it is an
eigenfunction of the {\em Laplace-Beltrami
operator} $\cl \equiv -\mathrm{div}\cdot\mathrm{grad}$
and $\phi(e)=1$. For $s\in\cc$, let $\phi_s$
be the {\em spherical function with eigenvalue} $s^2+n^2/4$.
It is known (\cite{A95}) that the spherical
function $\phi_0$ satisfies the estimate
\begin{equation}\label{phi0}
0<\phi_0(r)\ls (1+r) e^{-nr/2} \quad\quad\forall\ r>0,
\end{equation}
and that for every radial function $f\in C_c^\infty(S)$,
\begin{equation}\label{rho-formula}
\int_S \dz^{1/2}(x) f(x)\,d\rho(x)= \int_0^\infty \phi_0(r) f(r)A(r)\,dr.
\end{equation}

The {\em spherical Fourier transform} of a radial function $f$ in $L^1(\lambda)$
is defined
by the formula
$$\mathcal H f(s) \equiv  \int_S \phi_s(x) f(x)\,d\lz(x) \quad\quad
\forall \ s\in\cc.$$
For radial functions $f\in C_c^\infty(S)$, a {\em Plancherel formula} holds:
\begin{equation}\label{plancherel-formula}
\int_S |f(x)|^2\,d\lz(x)=C\int_0^\infty |\mathcal H f(s)|^2 |\mathbf c (s)|^{-2}\,ds,
\end{equation}
where $C$ is a positive constant depending only on $S$, and $|\mathbf c (s)|^{-2}\,ds$
denotes the {\em Plancherel measure} which satisfies the following estimates
{(see Chapter IV of \cite{helgason})}:
\begin{eqnarray}\label{plancherel-measure}
|\mathbf{c}(s)|^{-2} \le\left\{\begin{array}{ll}
C|s|^2\quad\quad&\text{if $|s|\le1$},\\
C|s|^n&\text{if $|s|>1$},
\end{array}\right.
\end{eqnarray}
where $C$ is a positive constant independent of $s$.
In particular, when $n=1$, the estimate \eqref{plancherel-measure} becomes
\begin{equation}\label{plancherel-measure'}
|\mathbf{c}(s)|^{-2} \le C \min\{|s|^2,\, |s|\} \qquad\forall\, s\in\mathbb \rr^+\,,
\end{equation}
where $C$ is a positive constant independent of $s$.

Denote by $\mathcal A$ the {\em Abel transform} and by $\mathcal A^{-1}$
the {\em inverse Abel transform}. If $n$ is even, then
\begin{equation}\label{abel-even}
\mathcal A^{-1} f(r) = (2\pi)^{-n/2} \lf(-\frac1{\sinh r}
\frac\partial {\partial r}\r)^{n/2} f(r)
\quad\quad\forall\ r>0,
\end{equation}
and if $n$ is odd, then for all $r>0$,
\begin{equation}\label{abel-odd}
\mathcal A^{-1} f(r) = (2\pi)^{-n/2} \int_r^\infty
\lf[\lf(-\frac1{\sinh s} \frac\partial {\partial s}\r)^{(n+1)/2} f(s)\r]
(\cosh s- \cosh r)^{-1/2} \sinh s\, ds.
\end{equation}
Denote by $\mathcal F (g)$ or $\widehat {g}$ the {\em Fourier transform}
of $g$ on $\rr$,
namely, $\mathcal F g(s)=\int_\rr g(r) e^{-isr}\,dr$.
It is known that $\mathcal H= \mathcal F \circ \mathcal A$, and hence
$\mathcal H^{-1}=\mathcal A^{-1}\circ \mathcal F^{-1}$.

Consider the following {\em basis of left-invariant vector fields of the Lie algebra of $S$}:
$$X_0 \equiv  a\partial_a,\quad
X_i \equiv  a\partial_{x_i}, \quad i=1,\, 2,\, \cdots,\, n.$$
The {\em Laplacian }$\Delta
\equiv  -\sum_{i=0}^nX_i^2$ is a left-invariant essentially selfadjoint operator on $L^2(\rho)$.
The operator $\Delta$ has a special relationship with the Laplace--Beltrami operator
$\cl$ associated with the Riemannian structure of $S$.
Indeed, if we denote by $\cl_n$ the {\em shifted  operator} $\cl-n^2/4$,
it is known that
\begin{equation}\label{L-BL}
\dz^{-1/2} \bigtriangleup \dz^{1/2} f = \cl_n f
\end{equation}
for all smooth radial functions $f$ on $S$ (see \cite{A95}), where $\delta$ denotes the {\em modular function}.
The {\em spectra} of both $\Delta$ on $L^2(\rho)$ and
$\cl_n$ on $L^2(\lz)$ are $[0, \infty)$.
Let $E_{\Delta}$ and $E_{\cl_n}$ be the {\em spectral
resolutions of identity} for which
$\bigtriangleup=\int_0^\fz t\,dE_{\bigtriangleup}(t)$ and
$\cl_n=\int_0^\infty t\, dE_{\cl_n}(t).$
For each bounded measurable function $m$ on $\rr^+$, the {\em operators} $m(\bigtriangleup)$
and $m(\cl_n)$, spectrally defined by
\begin{equation*}
m(\bigtriangleup)=\dint_0^\fz m(t)\,dE_\bigtriangleup (t) \quad \mbox{ and }\quad
m(\cl_n)=\dint_0^\fz m(t)\,dE_{\cl_n} (t),
\end{equation*}
are respectively bounded on $L^2(\rho)$ and $L^2(\lz)$ by the spectral theorem.
By \eqref{L-BL} and the spectral theorem,
we see that for all radial functions $f\in C_c^\infty(S)$,
\begin{equation*}
\dz^{-1/2} m(\bigtriangleup) \dz^{1/2} f= m(\cl_n) f.
\end{equation*}
Denote by $k_{m(\bigtriangleup)}$ the {\em convolution kernel} of
$m(\bigtriangleup)$, namely, for all $f\in C_c^\fz(S)$,
\begin{equation}\label{convolution-ker}
m(\bigtriangleup)f(x)=\dint_S f(xy^{-1})k_{m(\bigtriangleup)}(y)\,d\rho(y)
\quad \quad \forall\ x\notin\supp f.
\end{equation}
As in \eqref{convolution-ker}, denote by $k_{m(\cl_n)}$
the {\em convolution kernel} of $m(\cl_n)$.
It was proved in \cite{ADY, A95} that for all bounded measurable function $m$
on $\rr^+$, the convolution kernel $k_{m(\cl_n)}$ is radial,
\begin{equation}\label{sphericaltransform-kernel}
k_{m(\bigtriangleup)} =\dz^{1/2} k_{m(\cl_n)} \,\mbox{ and }\,
\mathcal H k_{m(\cl_n)}(s)= m(s^2) \quad\quad\forall\,\, s\in\rr^+.
\end{equation}
Let $K_{m(\bigtriangleup)}$ be the {\em integral kernel} of
$m(\bigtriangleup)$, namely, for all $f\in C_c^\fz(S)$,
\begin{equation}\label{integral-ker}
m(\bigtriangleup)f(x)=\dint_S K_{m(\bigtriangleup)}(x, y)f(y)\,d\rho(y)
\quad \quad \forall\ x\notin\supp f.
\end{equation}
In view of \eqref{convolution-ker} and \eqref{integral-ker},
by changing variables and using the left-invariant property of $\lz$
and the right-invariant property of $\rho$,
we obtain that for all $x$, $y\in S$,
\begin{equation}\label{kernelrelation}
K_{m(\bigtriangleup)}(x, y)=k_{m(\bigtriangleup)}(y^{-1}x)\dz(y).
\end{equation}

For any $s\in(0, \fz)$, we denote by $H^s(\rr)$ the {\em Sobolev
space} $W^{s, 2}(\rr)$ of order $s$ on $\rr$. Let $\phi\in
C_c^\fz(\rr^+)$ be a function supported in $[1/4, 4]$ such that
\begin{equation}\label{phi-decom}
\sum_{j\in\zz}\phi(2^{-j}t)=1
\quad\forall\ t\in\rr^+.
\end{equation}
For any given $s_0$, $s_\fz\in\rr^+$, a bounded measurable function
$m$ on $\rr^+$ is said to satisfy a {\em mixed Mihlin-H\"ormander
condition of order $(s_0, s_\fz)$} if
\begin{equation*}
\|m\|_{s_0} \equiv \sup_{t<1}\|m(t\cdot)\phi(\cdot)\|_{H^{s_0}(\rr)}<\fz
\ \mbox{ and }\
\|m\|_{s_\fz} \equiv \sup_{t\ge1}\|m(t\cdot)\phi(\cdot)\|_{H^{s_\fz}(\rr)}<\fz.
\end{equation*}
For any $j\in\zz$ and any
bounded measurable function $m$ on $\rr^+$, we
define $m_j$ by
\begin{equation}\label{m_j}
m_j(t) \equiv  m(2^jt)\phi(t)\qquad \forall \,t\in\rr^+.
\end{equation}
Obviously, we have
\begin{equation}\label{oper-decom}
m(\bigtriangleup)=\sum_{j\in\zz}m_j(2^{-j}\bigtriangleup).
\end{equation}
As in \eqref{convolution-ker} and \eqref{integral-ker}, we denote by
$k_{m_j(2^{-j}\bigtriangleup)}$ and $K_{m_j(2^{-j}\bigtriangleup)}$ the {\em convolution kernel} and the {\em integral kernel} of
$m_j(2^{-j}\bigtriangleup)$, respectively.

Assume that $m$ satisfies a mixed
Mihlin-H\"ormander condition of order $(s_0, s_\fz)$ with $s_0>3/2$
and $s_\fz>\max\{3/2, (n+1)/2\}$. Choose $\sz>0$ small enough such
that  $s_0>3/2+\sz$ and $s_\fz>\max\{3/2, (n+1)/2\}+\sz$.
Hebisch and Steger \cite[Theorems~2.4 and 6.1]{hs} proved that there
exists a positive constant $C$ such that for all $j\in\zz$ and $y\in S$,
\begin{eqnarray}\label{size-ker}
\dint_{S} |K_{m_j(2^{-j}\bigtriangleup)}(x, y)|(1+2^{j/2}d(x,
y))^\sz\, d\rho(x) \le \left\{\begin{array}{ll}
C\|m\|_{s_0}\quad\quad&\text{if $j\le0$},\\
C\|m\|_{s_\fz}&\text{if $j>0$};
\end{array}\right.
\end{eqnarray}
and that for all $y$, $z\in S$,
\begin{eqnarray}\label{hormander-ker}
&&\dint_{S} |K_{m_j(2^{-j}\bigtriangleup)}(x, y)-K_{m_j(2^{-j}
\bigtriangleup)}(x, z)|\, d\rho(x)\noz\\
&&\hs\le \left\{\begin{array}{ll}
C2^{j/2}d(y, z)\|m\|_{s_0}\quad&\text{if $j\le0$},\\
C2^{j/2}d(y, z)\|m\|_{s_\fz}&\text{if $j>0$}.
\end{array}\right.
\end{eqnarray}
From \eqref{size-ker} and \eqref{hormander-ker}, it is easy to deduce
that
\begin{equation}\label{hormander}
\sup_{R\in\crz}\sup_{y,\,z\in R}\dint_{S\setminus R^\ast}
|K_{m(\bigtriangleup)}(x, y)-K_{m(\bigtriangleup)}(x, z)|\,
d\rho(x)<\fz;
\end{equation}
see \cite[Remark 1.4]{hs}. From the proofs of \cite[Theorems~2.4 and 6.1]{hs},
it follows that \eqref{size-ker} and
\eqref{hormander-ker} still hold if we interchange the two variables of
$K_{m_j(2^{-j}\bigtriangleup)}$. Thus, $K_{m(\bigtriangleup)}$
satisfies H\"ormander's condition \eqref{4.2}.

The estimates \eqref{size-ker} and \eqref{hormander-ker} imply that
the operator
$m(\bigtriangleup)$ is bounded from $L^1$ to $L^{1,\fz}$
and bounded on $L^p$ for all $p\in(1,\fz)$ \cite[Theorem~2.4]{hs}
and that it is also bounded from $H^1$ to $L^1$ and
from $L^\fz$ to $\BMO$ \cite[Proposition~2.4]{va}.

We now consider the boundedness of the maximal singular integral operator
$\lf(m(\bigtriangleup)\r)^\ast$ as defined in \eqref{4.3}.

\begin{thm}\label{t5.1}
Let $m$ satisfy a mixed Mihlin-H\"ormander condition of order $(s_0,
s_\fz)$ with $s_0>3/2$ and $s_\fz>\max\{3/2, (n+1)/2\}$. Then the
maximal singular integral $\lf(m(\bigtriangleup)\r)^\ast$ is bounded
from $L^1$ to $L^{1,\fz}$, from $L_c^\fz$ to $\BMO$ and bounded
on $L^p$ for all $p\in(1, \fz)$.
\end{thm}

Theorem \ref{t5.1} follows immediately from Theorem \ref{t4.2} if
we know that $K_{m(\bigtriangleup)}$ satisfies
\eqref{4.1} and \eqref{4.2}. We already noticed that the kernel
$K_{m(\bigtriangleup)}$ satisfies H\"ormander's condition \eqref{4.2}.
Condition ({4.1}) for $K_{m(\bigtriangleup)}$ is equivalent
to the following estimate:
\begin{equation}\label{size-convolution-ker}
\sup_{\ez>0}\int_{\ez<d(y, e)\le 2\ez} |k_{m(\bigtriangleup)}(y)|\,d\rho(y) <\infty.
\end{equation}
To see this, by using the right-invariant property of $\rho$ and
the left-invariant property of $\lz$, we have that for all $x,y\in S$,
$$\dz(y)d\rho(y)=d\lz(y)=d\rho(y^{-1})=d\rho(y^{-1}x)\,\,\mbox{ and }
\,\, d(x, y)=d(y^{-1}x, e).$$
We then apply  \eqref{kernelrelation} to obtain
\begin{eqnarray*}
\int_{\ez<d(x, y)\le 2\ez} |K_{m(\bigtriangleup)}(x, y)|\,d\rho(y)
&& = \int_{\ez<d(x, y)\le 2\ez} |k_{m(\bigtriangleup)}(y^{-1}x)|\dz(y)\,d\rho(y)\\
&&= \int_{\ez<d(y^{-1}x, e)\le 2\ez} |k_{m(\bigtriangleup)}(y^{-1}x)|
\,d\rho(y^{-1}x)\\
&&= \int_{\ez<d(y, e)\le 2\ez} |k_{m(\bigtriangleup)}(y)|\,d\rho(y).
\end{eqnarray*}
Similarly, by \eqref{kernelrelation} and
$$\dz(y)d\rho(x)=\dz(y)[\dz(x)]^{-1}d\lz(x)=[\dz(y^{-1}x)]^{-1} d\lz(x)
=[\dz(y^{-1}x)]^{-1} d\lz(y^{-1}x)=d\rho(y^{-1}x),$$
we have
\begin{eqnarray*}
\int_{\ez<d(x, y)\le 2\ez} |K_{m(\bigtriangleup)}(x, y)|\,d\rho(x)
&& = \int_{\ez<d(x, y)\le 2\ez} |k_{m(\bigtriangleup)}(y^{-1}x)|\dz(y)\,d\rho(x)\\
&&= \int_{\ez<d(y^{-1}x, e)\le 2\ez} |k_{m(\bigtriangleup)}(y^{-1}x)|
\,d\rho(y^{-1}x)\\
&&= \int_{\ez<d(y, e)\le 2\ez} |k_{m(\bigtriangleup)}(y)|\,d\rho(y).
\end{eqnarray*}
Therefore, it suffices to show \eqref{size-convolution-ker}.
To this end, we will use some ideas from the proof of \cite[Theorem~4.3]{va2}.
Let us start with an integral estimate of the kernel which
is more delicate than the one proved in \cite[Lemma~4.5]{va2}.

\begin{lem}\label{l5.1}
Let $f$ be a bounded even function on $\rr$ such
that $\supp \widehat f \subset [-r, r]$.
Then, there exists a positive constant $C$ independent of $r$ such
that $k_{f(\sqrt\bigtriangleup)}$ satisfies the following:
\begin{enumerate}
\vspace{-0.2cm}
 \item[\rm(i)] if $\ez\ge r$, then $\int_{{\ez<d(x, e)\le 2\ez}}
|k_{f(\sqrt\bigtriangleup)}(x)|\, d\rho(x)=0$; \vspace{-0.25cm}
 \item[\rm(ii)] for all $\ez\in(0, 1)$ such that $\ez<r$,
\begin{equation}\label{small-epsilon}
\dint_{{\ez<d(x, e)\le 2\ez}} |k_{f(\sqrt\bigtriangleup)}(x)|\,
d\rho(x) \le
C\ez^{(n+1)/2}\lf\{\dint_0^\fz|f(s)|^2(s^2+s^n)\,ds\r\}^{1/2};
\end{equation}
\vspace{-0.3cm}
 \item[\rm(iii)] for all $\ez\in[1, \fz)$ such that  $\ez<r$,
\begin{equation}\label{large-epsilon}
\dint_{{\ez<d(x, e)\le 2\ez}} |k_{f(\sqrt\bigtriangleup)}(x)|\,
d\rho(x) \le
C\ez^{3/2}\lf\{\dint_0^\fz|f(s)|^2(s^2+s^n)\,ds\r\}^{1/2};
\end{equation}
\vspace{-0.3cm}
\item[\rm(iv)] when $n=1$, the right-hand sides of \eqref
{small-epsilon} and \eqref{large-epsilon}
can be  respectively replaced by the better estimates:
$C\ez^{(n+1)/2}\lf\{\int_0^\fz|f(s)|^2\min\{s, s^2\}\,ds\r\}^{1/2}$ and
$$C\ez^{3/2}\lf[\int_0^\fz|f(s)|^2\min\{s, s^2\}\,ds\r]^{1/2}.$$
\end{enumerate}
\end{lem}

\begin{proof}
Let $k_{f(\sqrt{\cl_n}\,)}$ be the {\em convolution kernel}
of the operator $f(\sqrt{\cl_n}\,)$.
By \eqref{sphericaltransform-kernel}, we know that
$k_{f(\sqrt{\cl_n}\,)}$ is radial on $S$,
$k_{f(\sqrt \triangle\,)}= \dz^{1/2} k_{f(\sqrt{\cl_n}\,)}$
and ${\mathcal H } k_{f(\sqrt{\cl_n}\,)} (t) = f(t)$ for all $t\in\rr^+.$
Since $\mathcal H^{-1}=\mathcal A^{-1}\circ \mathcal F^{-1}$ and $f$ is even,
we obtain
\begin{eqnarray}\label{e1}
k_{f(\sqrt{\cl_n}\,)} (t) = { \mathcal H}^{-1} f(t)
= {\mathcal A}^{-1} { \mathcal F}^{-1} f (t)={\mathcal A}^{-1} (\widehat f\,)(t).
\end{eqnarray}
Notice that $\supp \widehat f \subset[-r, r]$ and
the inverse formulae for the Abel transform \eqref{abel-even} and \eqref{abel-odd}
imply that {$\supp {\mathcal A}^{-1} (\widehat f\,) \subset B(e,r)$.}
Then, $\supp k_{f(\sqrt{\cl_n}\,)} \subset B(e, r)$, and hence $\supp k_{f(\sqrt \triangle\,)}\subset B(e, r)$.
So the integral of $ k_{f(\sqrt \triangle\,)}$ in the domain
$\{x\in S:\, \ez<d(x, e)\le2\ez\}$
is $0$ when $\ez\ge r$. This proves (i).

To prove (ii) and (iii), we set
$w(x) \equiv  \dz^{-1/2}(x) e^{nd(x, e)/2}$ for all $x\in S$.
Applying H\"older's inequality yields that
\begin{eqnarray*}
\quad
&&\dint_{{\ez<d(x, e)\le 2\ez}} |k_{f(\sqrt\bigtriangleup\,)}(x)|\,
d\rho(x)\\
&&\quad\le \lf\{
\dint_{{\ez<d(x, e)\le 2\ez}} w(x)^{-1}\,
d\rho(x)
\r\}^{1/2}\lf\{\dint_{{\ez<d(x, e)\le 2\ez}} |k_{f(\sqrt\bigtriangleup\,)}(x)|^2\,
w(x)\,d\rho(x)\r\}^{1/2}\nonumber\\
&&\hs \equiv  {\rm I}^{1/2} \cdot {\rm J}^{1/2},
\end{eqnarray*}
where we denoted by ${\rm I}$ and ${\rm J}$ respectively the integral in the
first and second bracket.
Recall that if $x=(y, a)\in S$ with $y\in\mathbb R^n$ and
$a\in\mathbb R^+$, then $\delta(x)=\delta(y,a)=a^{-n}$.
When $\ez\in(0, 1]$, for all $d(x,e)\le 2\ez$, we have $|\dz(x)|\ls \ez\ls1$, and hence
\begin{eqnarray*}
{\rm I} &&=\dint_{{\ez<d(x, e)\le 2\ez}} \dz^{1/2}(x) e^{-nd(x, e)/2}\,
d\rho(x)\ls\rho\big(B(e, 2\ez)\big) \ls \ez^{n+1}.
\end{eqnarray*}
When $\ez>1$, by \eqref{rho-formula} together with the
estimates \eqref{phi0} and \eqref{misA} of $\phi_0$ and
of the density function $A$, we obtain
\begin{eqnarray*}
{\rm I}&&\leq \int_{d(x, e)\le 2\ez} \dz^{1/2}(x) e^{-nd(x, e)/2} d\rho(x)\\
&&=\int_0^{2\ez} \phi_0(t) e^{-nt/2} A(t)\,dt
\ls\int_0^{2\ez} (1+t)\bigg(\frac t{1+t}\bigg)^{n}\,dt\ls\ez^2.
\end{eqnarray*}
To estimate ${\rm J}$, since $k_{f(\sqrt \triangle\,)}=
\dz^{1/2} k_{f(\sqrt{\cl_n}\,)}$, we have
\begin{eqnarray*}
{\rm J}&&= \dint_{{\ez<d(x, e)\le 2\ez}} \dz(x)|k_{f(\sqrt\cl_n\,)}(x)|^2\,
\dz^{-1/2}(x) e^{nd(x, e)/2} \,d\rho(x).
\end{eqnarray*}
Again, using \eqref{rho-formula} and the estimates \eqref{phi0} and \eqref{misA},
we  estimate $\rm J$ by
\begin{eqnarray*}
{\rm J}&&= \int_\ez^{2\ez} \phi_0(t) |k_{f(\sqrt\cl_n\,) }(t)|^2 e^{nt/2} A(t)\,dt\\
&&\ls (1+\ez) \int_\ez^{2\ez} |k_{f(\sqrt\cl_n\,) }(t)|^2  A(t)\,dt
\ls (1+\ez) \int_S |k_{f(\sqrt\cl_n\,) }(x)|^2\,d\lz(x),
\end{eqnarray*}
where the last inequality is due to \eqref{lambda-formula}.  Applying the Plancherel
formula \eqref{plancherel-formula}
and the estimate for the Plancherel measure \eqref{plancherel-measure}
(when $n=1$ we use \eqref{plancherel-measure'} instead)
yields that
\begin{eqnarray}\label{e2}
\int_S |k_{f(\sqrt\cl_n\,) }(x)|^2\,d\lz(x)
&&\approx\int_0^\infty |{\mathcal H}k_{f(\sqrt\cl_n\,)} (t) |^2 |{\mathbf c} (t)|^{-2}\,dt\nonumber\\
&&\approx\int_0^\infty |f(t)|^2 |{\mathbf c} (t)|^{-2}\,dt\ls\int_0^\infty |f(t)|^2 (t^2+t^n)\,dt,
\end{eqnarray}
which implies that
${\rm J}\ls (1+\ez) \int_0^\infty |f(t)|^2 (t^2+t^n)\,dt.$
Combining the estimate of ${\rm I}$ and ${\rm J}$ yields (ii) and (iii).

The proof for (iv) follows from the same argument
except that in \eqref{e2} we use \eqref{plancherel-measure'}
instead of \eqref{plancherel-measure}.
\end{proof}

The following decomposition of functions with compact support was
proved in \cite[Lemma~1.3]{H98}; see also
\cite[Lemma~4.6]{va2}.

\begin{lem}\label{l5.2}
Let $q,\,Q\in[0, \infty)$. Suppose that $f\in H^s(\rr)$ and
$\supp f\subset[1/2, 2]$. Then there exist even functions
$\{f_\ell\}_{\ell=0}^\fz$ and a positive constant $C$, independent of $f$ and $\ell$,
such that for all $\tau\in\rr^+$,
\begin{enumerate}
\vspace{-0.2cm}
 \item[(i)] $f(\tau\cdot)=\sum_{\ell=0}^\fz f_{\ell, \tau}(\cdot)$ on $\rr^+$,
where $f_{\ell, \tau}(\cdot) \equiv  f_\ell(\tau\cdot)$ and
 $\supp \widehat {f_{\ell, \tau}}\subset [-2^\ell \tau, 2^\ell \tau]$;
\vspace{-0.25cm}
 \item[(ii)] for all $\ell\ge0$ and $\tau\in[1, \fz)$,
\begin{equation}\label{large-dilation}
\int_0^\fz|f_{\ell, \tau}(\xi)|^2(\xi^{2q}+\xi^{2Q})\,d\xi
\le C \tau^{-2\min\{q, Q\}-1}2^{-2s\ell}\|f\|_{H^s(\rr)}^2;
\end{equation}
\vspace{-0.5cm}
\item[(iii)] for all $\ell\ge0$ and $\tau\in(0,1)$,
\begin{equation}\label{small-dilation}
\int_0^\fz|f_{\ell, \tau}(\xi)|^2(\xi^{2q}+\xi^{2Q})\,d\xi
\le C \tau^{-2\max\{q, Q\}-1}2^{-2s\ell}\|f\|_{H^s(\rr)}^2.
\end{equation}
\vspace{-0.5cm}
\end{enumerate}
\end{lem}

\begin{prop}\label{p5.1} Let $m$ satisfy a mixed Mihlin-H\"ormander condition of order $(s_0,
s_\fz)$ with $s_0>3/2$ and $s_\fz>\max\{3/2, (n+1)/2\}$.
Then $k_{m(\bigtriangleup)}$ satisfies \eqref{size-convolution-ker}.
\end{prop}

\begin{proof}
Let $m_j$ be as in \eqref{m_j}. By \eqref{oper-decom}, we obtain that for
all $x\in S$ and $\ez>0$,
\begin{eqnarray*}
&&\dint_{\ez<d(x, e)\le2\ez}\lf|k_{m(\bigtriangleup)}(x)\r|\,d\rho(x)\\
&&\hs=\sum_{\{j\in\zz:\,2^{j/2}\ez\ge1\}}
\dint_{\ez<d(x, e)\le2\ez}\lf|k_{m_j(2^{-j}\bigtriangleup)}(x)\r|\,d\rho(x)\\
&&\hs\quad+
\sum_{\{j\in\zz:\,2^{j/2}\ez<1\}} { \dint_{\ez<d(x, e)\le2\ez} }
\lf|k_{m_j(2^{-j}\bigtriangleup)}(x)\r|\,d\rho(x)\equiv  {\rm I} +{\rm J}.
\end{eqnarray*}
Observe that, by \eqref{kernelrelation}, $K_{m_j(2^{-j}\bigtriangleup)}(x, e)=k_{m_j(2^{-j}\bigtriangleup)}(x)$
for all $x\in S$.  From this and \eqref{size-ker}, it follows that
\begin{eqnarray*}
{\rm I}&&= \sum_{\{j\in\zz:\,2^{j/2}\ez\ge1\}}
\dint_{\ez<d(x, e)\le2\ez}\lf|K_{m_j(2^{-j}\bigtriangleup)}(x, e)\r|\,d\rho(x)\\
&&\ls\sum_{\{j\in\zz:\,2^{j/2}\ez\ge1\}}
\dfrac{1}{(1+2^{j/2}\ez)^\sz} \dint_{\ez<d(x, e)\le2\ez}
|K_{m_j(2^{-j}\bigtriangleup)}(x, e)| (1+2^{j/2}d(x, e))^\sz\,
d\rho(x)\ls1.
\end{eqnarray*}
To estimate ${\rm J}$, for each $j\in\zz$, set $f^{(j)}(t) \equiv  m_j(t^2)$ for all $t\in
\rr^+$. Since $\supp f^{(j)}\subset [1/2, 2]$, we use Lemma
\ref{l5.1} to decompose each $f^{(j)}$. Hence, there exists a
sequence of even functions $\{f^{(j)}_\ell\}_{\ell=0}^\infty$ such
that $f^{(j)}(2^{-j/2}\cdot)=\sum_{\ell=0}^\fz f^{(j)}_{\ell,
2^{-j/2}}(\cdot)$ on $\rr^+$ and the Fourier transform of each $f^{(j)}_{\ell,
2^{-j/2}}$ is supported in $[-2^{\ell-j/2}, 2^{\ell-j/2}]$. For
simplicity, in the sequel, we denote $f^{(j)}_{\ell, 2^{-j/2}}$ by
$h_{\ell, j}$. Notice that
$m_j(2^{-j}\bigtriangleup)=f^{(j)}(2^{-j/2}\sqrt\bigtriangleup)$.
Summarizing all these facts, we obtain that for all $j\in\zz$,
\begin{eqnarray}\label{decom-convolution-ker}
\dint_{\ez<d(x,
e)\le2\ez}\lf|k_{m_j(2^{-j}\bigtriangleup)}(x)\r|\,d\rho(x)
&&=\sum_{\ell=0}^\fz\dint_{\ez<d(x, e)\le2\ez} \lf|k_{h_{\ell,
j}(\sqrt\bigtriangleup)}(x)\r|\,d\rho(x).
\end{eqnarray}
By the support condition of $\widehat{h_{\ell, j}}$,
we have $\supp k_{h_{\ell, j}(\sqrt\bigtriangleup)}
\subset B(e, 2^{\ell-\frac j2})$. By this and Lemma
\ref{l5.2}(i), the sum  in
\eqref{decom-convolution-ker} reduces to the sum
for $\ell$ satisfying $2^{\ell-j/2}>\ez$,
therefore,
$${\rm J}=\sum_{\{j\in\zz:\,2^{j/2}\ez<1\}}\
\sum_{\{\ell\in\nn:\,2^{\ell-j/2}>\ez\}} \dint_{\ez<d(x,
e)\le2\ez}\lf|k_{h_{\ell, j}(\sqrt\bigtriangleup)}(x)\r|\,d\rho(x).$$
Set
$$ {\rm W}_1 \equiv \{(j, \ell):\, 2^{j/2}\ez<1,\,2^{\ell-j/2}>\ez,\,
\ell\ge0,\, j>0,\, 2^{\ell-j/2}\ge 1\},$$
$${\rm W}_2 \equiv \{(j, \ell):\, 2^{j/2}\ez<1,\,2^{\ell-j/2}>\ez,\,
\ell\ge0,\, j>0,\, 2^{\ell-j/2}<1\},$$
$${\rm W}_3 \equiv \{(j, \ell):\, 2^{j/2}\ez<1,\,2^{\ell-j/2}>\ez,\,
\ell\ge0,\, j\le0,\, 2^{\ell-j/2}\ge 1\}$$
and
$${\rm W}_4 \equiv \{(j,
\ell):\, 2^{j/2}\ez<1,\,2^{\ell-j/2}>\ez,\, \ell\ge0,\, j\le0,\,
2^{\ell-j/2}< 1\}.$$
Observe
that ${\rm W}_4=\emptyset$.
Correspondingly, for all $k=1, 2, 3$, we set
$${\rm J}_k \equiv \sum_{(j, \ell)\in {\rm W}_k} \dint_{\ez<d(x, e)\le2\ez}
\lf|k_{h_{\ell, j}(\sqrt\bigtriangleup)}(x)\r|\,d\rho(x).$$  Thus, ${\rm J}=\sum_{k=1}^3{\rm J}_k$.

If $(j, \ell)\in {\rm W}_1$, then $j>0$ and hence $\ez<2^{-j/2}<1$. By
this, \eqref{small-epsilon}, \eqref{small-dilation} and the
assumption $s_\fz>\max\{3/2,
(n+1)/2\}$, we obtain
\begin{eqnarray*}
{\rm J}_1 &&\ls \sum_{j>0}\sum_{\{\ell\in\nn:\,2^{\ell-j/2}\ge
1\}}
\ez^{(n+1)/2} \lf\{\dint_0^\fz|h_{\ell, j}(t)|^2(t^2+t^n)\,dt\r\}^{1/2}\\
&&\ls \sum_{{j>0}}\sum_{\ell\ge j/2}
2^{\frac j2\max\{3/2, (n+1)/2\}} 2^{-s_\fz\ell}\|f^{(j)}\|_{H^{s_\fz}(\rr)}\\
&&\ls \sum_{j>0}\sum_{\ell\ge j/2}2^{\frac
j2(\max\{3/2, (n+1)/2\}-s_\fz)} 2^{-s_\fz(\ell-j/2)}
\|m\|_{s_\fz}\ls\|m\|_{s_\fz}.
\end{eqnarray*}

If $(\ell, j)\in {\rm W}_2$, then $\ez\in(0, 1)$. Applying
\eqref{small-epsilon} and \eqref{small-dilation}
yields that when $n\ge2$,
\begin{eqnarray*}
{\rm J}_2 &&\ls
\sum_{j=1}^{-2\log_2\ez}\sum_{\ell<j/2}
\ez^{(n+1)/2}\lf\{\dint_0^\fz|h_{\ell, j}(t)|^2(t^2+t^n)\,dt\r\}^{1/2}\\
&&\ls \sum_{j=1}^{-2\log_2\ez}\sum_{\ell=0}^\fz
\ez^{(n+1)/2}2^{\frac j2\max\{3/2,
(n+1)/2\}}2^{-s_\fz\ell}\|m\|_{s_\fz} \ls\|m\|_{s_\fz};
\end{eqnarray*}
when $n=1$, by Lemma \ref{l5.1}(iv),
 we replace
 $\{\int_0^\fz|h_{\ell, j}(t)|^2(t^2+t^n)\,dt\}^{1/2}$
in the above estimate by $\{\int_0^\fz|h_{\ell, j}(t)|^2t\,dt\}^{1/2}$,
and then a similar argument also implies that ${\rm J}_2\ls\|m\|_{s_\fz}$.

Now we assume $n\ge2$ and estimate ${\rm J}_3$. In this case, when
$\ez\in[1, \fz)$, we use \eqref{large-epsilon}, \eqref{large-dilation} and
$2^{j/2}\ez<1$ to obtain
\begin{eqnarray}\label{J3-epsilon-large}
{\rm J}_3 &&\ls \sum_{\{j\le0:\,2^{j/2}\ez<1\}}\
\sum_{\ell=0}^\fz\ez^{3/2}
\lf\{\dint_0^\fz|h_{\ell, j}(t)|^2(t^2+t^n)\,dt\r\}^{1/2}\noz\\
&&\ls \sum_{\{j\le0:\,2^{j/2}\ez<1\}}\ \sum_{\ell=0}^\fz\ez^{3/2}
2^{3j/4} 2^{-s_0\ell}\|m\|_{s_0}\ls \|m\|_{s_0}.
\end{eqnarray}
When $\ez\in(0,1)$, applying \eqref{large-epsilon} and
\eqref{large-dilation} yields that
\begin{eqnarray}\label{J3-epsilon-small}
{\rm J}_3&&\ls \sum_{j=-\fz}^0 \sum_{\ell=0}^\fz\ez^{(n+1)/2}
\lf\{\dint_0^\fz|h_{\ell, j}(t)|^2(t^2+t^n)\,dt\r\}^{1/2}\noz\\
&&\ls \sum_{j=-\fz}^0 \sum_{\ell=0}^\fz2^{3j/4}
2^{-s_0\ell}\|m\|_{s_0} \ls \|m\|_{s_0}.
\end{eqnarray}
If $n=1$, Lemma \ref{l5.1}(iv) implies that
we can replace $\{\int_0^\fz|h_{\ell, j}(t)|^2(t^2+t^n)\,dt\}^{1/2}$
in \eqref{J3-epsilon-large} and \eqref{J3-epsilon-small}
with $\{\int_0^\fz|h_{\ell, j}(t)|^2t^2\,dt\}^{1/2}$,
and a similar argument also yields ${\rm J}_3\ls\|m\|_{s_0}$.

Combining the estimates of ${\rm J}_1$ through ${\rm J}_3$ yields
${\rm J}\ls1$. From this and the estimate of ${\rm I}$, we deduce
that $k_{m(\bigtriangleup)}$ satisfies \eqref{size-convolution-ker},
which completes the proof.
\end{proof}

\begin{proof}[Proof of Theorem \ref{t5.1}]
By Proposition \ref{p5.1},
the $L^2$-boundedness of $m(\Delta)$ and the fact that the kernel of $m(\Delta)$ satisfies H\"ormander's
condition \eqref{4.2}, the operator $m(\Delta)$ satisfies all the assumptions of Theorem \ref{t4.2}.
Then $(m(\bigtriangleup))^\ast$ satisfies the desired properties.
\end{proof}

\medskip

{\small\noindent{\bf Acknowledgements}\quad The authors would like to
thank the referee very much for her/his very carefully reading
and many valuable remarks which made this article more readable.}

\bigskip

{\sc Liguang Liu}

School of Mathematical Sciences, Beijing Normal University,
Laboratory of Mathematics and Complex Systems, Ministry of
Education, Beijing 100875, China

\smallskip

({\it Present address}) Department of Mathematics,
School of Information,
Renmin University
of China, Beijing 100872, People's Republic of China

{\it E-mail}: \texttt{liuliguang@ruc.edu.cn}

\medskip

{\sc Maria Vallarino}

Dipartimento di Matematica,
Politecnico di Torino,
Corso Duca degli Abruzzi 24,
10129 Torino,
Italy

{\it E-mail}: \texttt{maria.vallarino@polito.it}

\medskip

{\sc Dachun  Yang} (Corresponding author)

School of Mathematical Sciences, Beijing Normal University,
Laboratory of Mathematics and Complex Systems, Ministry of
Education, Beijing 100875, China

{\it E-mail}: \texttt{dcyang@bnu.edu.cn}

\end{document}